\documentclass[reqno,11pt]{amsart}

\usepackage{amsmath,amssymb}
\usepackage{graphicx,epsfig,color}

\newtheorem{theorem}{Theorem}
\newtheorem{lemma}{Lemma}

\newtheorem{proposition}{Proposition}
\newtheorem{remark}{Remark}

\newcommand{\eps}{\varepsilon}

\title[A discontinuous Poisson--Boltzmann equation]
{A discontinuous Poisson--Boltzmann 
equation with interfacial transfer: \\
homogenisation and residual error estimate}

\author[Klemens Fellner and Victor A. Kovtunenko]{Klemens Fellner and Victor A. Kovtunenko}

\address{Klemens Fellner \hfill\break
Institute of Mathematics and Scientific Computing, University of Graz, 
NAWI Graz, Heinrichstra{\ss}e 36, 8010 Graz, Austria}
\email{klemens.fellner@uni-graz.at}

\address{Victor A. Kovtunenko\hfill\break
Institute of Mathematics and Scientific Computing, University of Graz, 
NAWI Graz, Heinrichstra{\ss}e 36, 8010 Graz, Austria and
Lavrent'ev Institute of Hydrodynamics, 630090 Novo\-sibirsk, Russia}
\email{victor.kovtunenko@uni-graz.at}

\begin{document}



\begin{abstract}
A nonlinear Poisson--Boltzmann equation with transmission boundary conditions 
at the interface between two materials is investigated. 
The model describes the electrostatic potential generated by a vector of ion concentrations 
in a periodic multiphase medium with dilute solid particles. 

The key issue is that the interfacial transfer allows jumps and 
thus discontinuous solutions of the problem. 
Based on variational techniques, we derive the homogenisation of the discontinuous problem 
subject to inhomogeneous transmission interface conditions. Moreover, we establish a rigorous residual error 
estimate up to the first order correction. 
\end{abstract}

\maketitle

\section{Introduction}\label{sec1}

In this paper, we consider the steady state problem of a nonlinear Poisson--Nernst--Planck (PNP) system, which describes multiple concentrations of charged particles (e.g. ions) subject to a self-consistent electrostatic potential calculated 
from Poisson's equation. In particular, we shall investigate the PNP model on a multiphase medium. The prototypical 
multiphase medium in mind consists of an electrolyte medium, which surrounds disjoint solid particles.  
Such models have numerous applications describing electro-kinetic phenomena in bio-molecular or electro-chemical models, photo-voltaic systems and semiconductors, 
see e.g. \cite{AMP/10,ACDDJLMTV/04,BFMW/13,Hum/00,LC/06,RMK/12,SB} 
and references therein. 
Our specific interests are motivated by models of Li-Ion batteries, see e.g. \cite{Pic/11}. 

In order to be able to deal with the nonlinearity of the model, we shall work within an analytic framework, 
where the PNP system can be equivalently transformed into 
a scalar semi-linear Poisson--Boltzmann (PB) equation. 
This is possible, when  reaction terms in the charged particle fluxes are omitted and the equations for the concentrations decouple 
since the charged particle concentrations are explicitly determined 
by the corresponding Boltzmann statistics. 
For references applying linearisation of the PNP equations near the Boltzmann distribution see e.g. \cite{AMP/10,LC/06}.

The major difficulty addressed in this manuscript is 
the imposed inhomogeneous intermedia transmission 
boundary condition for the electrostatic field, which complements the PB equation (see \eqref{2.8} below). 
Thus, the key feature of the presented model is the electric charge 
transport phenomena over the interfaces at the boundaries of the solid particles. 
The interfacial transfer shall be described by the Gouy--Chapman-Stern model 
for electric double layers (EDLs) \cite{Pic/11}.
This model proposes a jump of the electrostatic field across the interface (a voltage drop) 
as well as a current prescribed at the interior boundary of the solid particles. 

In the following, we will derive a discontinuous formulation of the PB equation 
(valid both on the volume occupied by the solid particles and on the surrounding porous space) 
with inhomogeneous transmission conditions at the interfaces between particles and porous space. 

A first aim of this paper is to establish a proper variational setting 
of the transmission problem, while a second part deals with its rigorous homogenisation.  
In respect to the later, we emphasise that 
the averaged effective coefficients of the limit problem represent the
macroscopic behaviour of the EDL, which is of primary practical importance. 

For reference concerning the classic homogenisation theories, we refer to 
\cite{Arg/04,BP/89,BLP/11,OSY/92,SP/80,ZKO/94}. 
The applied methods range from two-scale convergence (see e.g. \cite{All/92}) over Gamma-convergence (see e.g. \cite{DM/06}) to unfolding (see \cite{CDDGZ/12}) and others. 
While formal methods of averaging are widely used in the literature, 
their verification in terms of residual error estimates 
is a hard task. 

From the point of view of homogenisation, the principal difficulty of 
interfacial transmission problems concerns the non-standard boundary conditions with jumps: 
On the one hand, related jump conditions are inherent for cracks. 
For models and methods used in crack problems, we refer to 
\cite{FHLRS/09,HKK/09,KK/00,SP/80} and references therein. 
From a geometric viewpoint, cracks are open manifolds in the reference domain. 
Hence, classic Poincare--Friedrichs--Korn inequalities are valid in such situations. 
In contrast to cracks, the interfaces here are assumed to be closed manifolds 
disconnecting the reference domain. 
This difference requires discontinuous versions of Poincare--Friedrichs--Korn 
inequalities, which are then applied for semi-norm estimates. 

On the other hand, the transmission boundary conditions are of Robin type. 
The homogenisation results known for linear problems with Robin (also called Fourier) conditions are crucially sensitive to the asymptotic rates of the involved homogenisation parameters. 
This issue concerns the coefficients in the boundary condition (cf. Lemma \ref{lem3.1} below) and the volume fraction 
of solid particles in periodic cells (cf. Lemma \ref{lem3.2} below), see e.g. \cite{ADH/96,BPC/01,OS/96}. 

The literature on homogenisation of transmission problems is very scarce, 
see e.g. \cite{Hum/00,Orl/12}. 
The technical challenge of this manuscript 
is the combination of nonlinearity, discontinuity and Robin type transmission 
conditions.
\smallskip

In the present work, we homogenise the discontinuous nonlinear PB equation with 
inhomogeneous interfacial transfer conditions and derive the averaged limit problem.
A further major result is the rigorous derivation of the residual error up to  
the first order correction. 

For these purposes, we develop a variational technique based on orthogonal Helmholtz 
decomposition following the lines of \cite{OSY/92,ZKO/94}. 
In a periodic cell, we decompose oscillating coefficients (describing the 
electric permittivity) by using the nontrivial kernel in the space of vector valued 
periodic functions, which is represented by sums of constant and divergence free (and thus, skew symmetric) vector fields
(cf. Lemma \ref{lem3.3}).  
Employing solutions of appropriately defined discontinuous cell problems, 
we obtain a regular decomposition of the homogenisation problem (see Theorem \ref{theo3.1}). 

A second result establishes the critical rates of the asymptotic behaviour 
with respect to a homogenisation parameter $\eps\searrow0^+$ 
for coefficients in the inhomogeneous transmission condition: 
We find on the one side that the critical rate for the coefficient 
by interfacial jumps is $\frac{1}{\eps}$. 
This factor occurs in the discontinuous Poincare inequality 
(for the norm squared, cf. \eqref{2.19} below)
and is thus relevant for a coercivity estimate, which in return 
contributes to the solvability of the discontinuous problem 
and the subsequent estimate of the homogenisation error. 

On the other side, the critical rate for the flux prescribed 
at the interior boundary of solid particles is $\eps$. 
At this rate, the interior boundary flux induces an additional potential, which  
distributes over the macroscopic 
domain in the homogenisation limit $\eps\searrow0^+$. 
If the asymptotic rate is lower than the critical one, then this flux vanishes in the limit. 
Otherwise, if the asymptotic rate is bigger, then the flux term diverges. 

From the above description we summarise the key points of this paper as follows: 
\begin{itemize}
\item 
the study of inhomogeneous interfacial transfer conditions describing EDL; 
\item 
the combination of nonlinear terms, jumps and Robin conditions;
\item 
a variational framework of the transmission problem; 
\item 
the performing of the homogenisation procedure with rigorous error estimates; 
\item 
the identification of the critical asymptotic rates of the boundary coefficients.
\end{itemize}
\medskip

\noindent\underline{Outline}: In the Sections \ref{sec2.1}, \ref{sec2.2} and \ref{sec2.3}, we first present the problem geometry, the physical and the 
mathematical model. Section~\ref{sec2.3} establishes moreover the equivalence of the steady-state of the PNP model with the semi-linear Poisson-Boltzmann equation and the existence 
of a unique solution to the PB equation (see Theorem \ref{theo2.1}).

In Section \ref{sec3}, we consider the homogenisation problem and the 
residual error estimate. At first, we state three auxiliary Lemmata
before stating the main homogenisation Theorem \ref{theo3.1}.

Finally, Section \ref{sec4} provides a brief discussion of the obtained results.

\section{Statment of the Problem}\label{sec2}

We start with the description of the geometry. 

\subsection{Geometry}\label{sec2.1}\hfill\\
Let $\omega$ denote the domain occupied by solid particles  of general shape (either single or multiple particles), which are located inside the unit cell $\Upsilon=(0,1)^d\subset\mathbb{R}^d$, $d=1,2,3$. We assume that all 
particles $\omega\subset\Upsilon$ are disjunctively located as well as bounded away from the boundary $\partial\Upsilon$, i.e. $\omega\cap\partial\Upsilon=\emptyset$.  

We assume that the boundary $\partial\omega$ is Lipschitz continuous with 
outer normal vector $\nu =(\nu_1,\dots,\nu_d)^\top$ pointing away from the domain $\omega$. Moreover, we distinguish the positive (outward orientated) surface $\partial\omega^+$ and the negative (inward orientated) surface $\partial\omega^-$ as the faces of the boundary $\partial\omega$, when approaching the boundary $\partial\omega$ from outside, i.e. from $\Upsilon\setminus\omega$ or from the inside, i.e. from $\omega$, 
respectively. For a two-dimensional example configuration see the illustration 
in Fig.~\ref{fig1} (a).

\begin{figure}[hbt!]
\begin{center}
\hspace*{-2cm}
\scalebox{.7} {\includegraphics{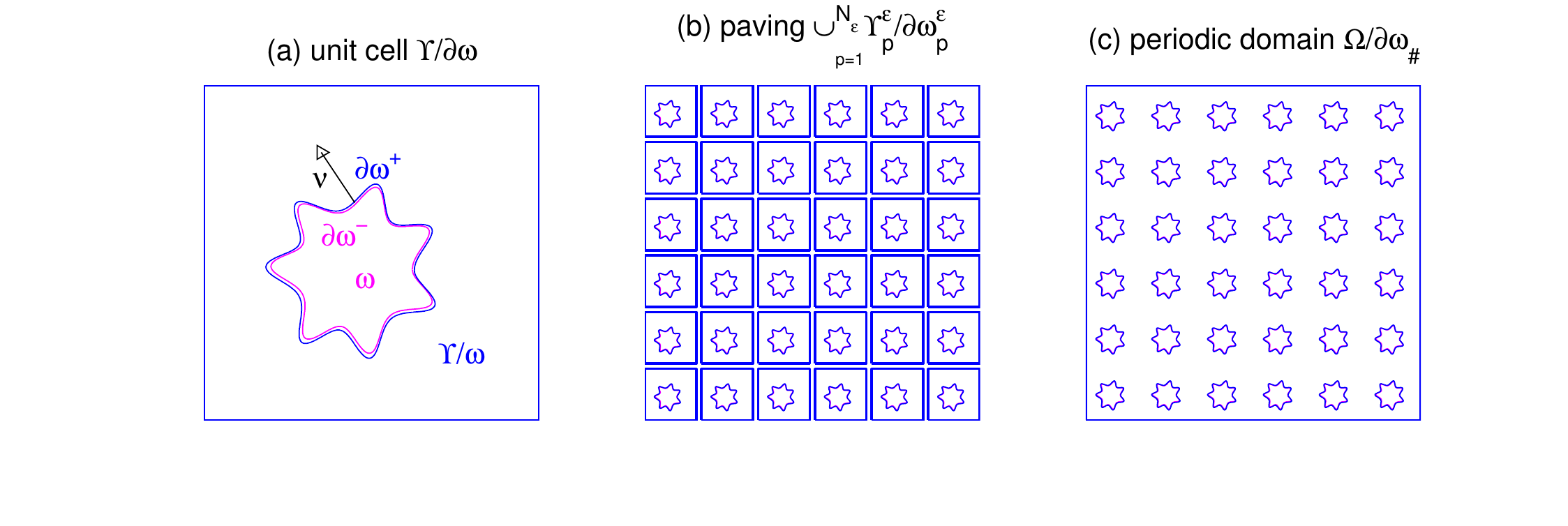}}
\vspace*{-1cm}
\caption{Two-dimensional example geometry with one star-shaped particle: (a) the unit cell, (b) the paving and (c)  the periodic disjoint domains $\Omega\setminus\partial\omega_\#$.}
\label{fig1}
\end{center}
\end{figure}

In the following, we consider a fixed, small homogenisation parameter $\eps\in\mathbb{R}_+$ and pave $\mathbb{R}^d$ with periodic cells $\Upsilon_p^\eps$ indexed by $p\in\mathbb{N}$. The periodic cells $\Upsilon_p^\eps$  
are constructed from $\Upsilon$ in the following way:
The position of every spatial point $x=(x_1,\dots,x_d)^\top\in\mathbb{R}^d$ 
can be decomposed as 
$$
x=\eps\left\lfloor\frac{x}{\eps}\right\rfloor +\eps\left\{\frac{x}{\eps}\right\}, 
\qquad\quad \left\lfloor\frac{x}{\eps}\right\rfloor\in\mathbb{Z}^d, \quad 
\left\{\frac{x}{\eps}\right\}\in\Upsilon,
$$ 
into the integer-valued floor function coordinates 
$\lfloor\frac{x}{\eps}\rfloor\in\mathbb{Z}^d$ and the fractional coordinates 
$\{\frac{x}{\eps}\}\in\Upsilon$. 
We shall then enumerate all possible integer vectors $\lfloor\frac{x}{\eps}\rfloor$ by means of a natural ordering with the index $p\in\mathbb{N}$. 
According to this index, we associate $\eps\lfloor\frac{x}{\eps}\rfloor$ with the $p$-th cell $\Upsilon_p^\eps$ 
and $\eps\{\frac{x}{\eps}\} =\eps y$ shall denote the local coordinates 
in all cells which correspond to $y\in\Upsilon$. 

We will denote by $\omega_p^\eps\subset\Upsilon_p^\eps$ the respective 
solid particles obtained by means of the paving with $\{\frac{x}{\eps}\} =y$ for $y\in\omega$. We note that the rescaling does not change the unit outer normal vector $\nu$. 

Evidently, the periodic mapping $x\mapsto y$, $\mathbb{R}^d\mapsto\Upsilon$, 
is surjective. 
This construction can be generalised to an arbitrary orthotope 
$\Upsilon$, see \cite{CDDGZ/12}. 
\medskip

Let $\Omega$ be the reference domain in $\mathbb{R}^d$ with Lipschitz 
boundary $\partial\Omega$ and denote again the outer normal vector by $\nu$. 
By reordering the index $p$, it is then possible to account for all solid particles $\omega_p^\eps\subset\Omega$ with the index set $p=1,\dots,N_\eps$, 
see \cite{CDDGZ/12,Fra/10}. We remark that $N_\eps\sim\eps^{-d}$. 

By omitting solid particles which are "too close" to the external boundary 
$\partial\Omega$, we shall ensure a constant gap with the distance $O(\eps)$ between $\partial\Omega$ and all particles $\omega_p^\eps$. 
Thus, $\Omega$ is divided into the multiple domains 
$\omega_\# :=\cup_{p=1}^{N_\eps} \omega_p^\eps$ corresponding to all the 
solid particles located periodically in the reference domain
and the remaining porous space $\Omega\setminus\omega_\#$. 

In the following, we shall denote by $\partial\omega_\# 
=\cup_{p=1}^{N_\eps} \partial\omega_p^\eps$ the union of 
boundaries $\partial\omega_p^\eps$ 
and introduce the disjoint multiple domains 
$$
\Omega\setminus\partial\omega_\# =(\Omega\setminus\omega_\#) 
\cup\omega_\#, \qquad\quad 
\partial\omega_\# =\cup_{p=1}^{N_\eps} \partial\omega_p^\eps, \qquad 
\omega_\# :=\cup_{p=1}^{N_\eps} \omega_p^\eps.
$$

Moreover, for functions $\xi$, which are discontinuous over the 
interface $\partial\omega_\#$, we will denote the jump across the interface by 
$$
[\![\xi]\!] :=\xi^+-\xi^-, \qquad \xi^\pm :=\xi|_{\partial\omega_\#^\pm}.
$$
Here, $\partial\omega_\#^+ =\cup_{p=1}^{N_\eps} (\partial\omega_p^\eps)^+$ 
summarises the positive faces (orientated towards the interior of the pore space $\Omega\setminus\omega_\#$), 
and $\partial\omega_\#^- =\cup_{p=1}^{N_\eps} (\partial\omega_p^\eps)^-$ 
accounts for the negative faces (orientated towards the interior of the solid phase 
$\omega_\#$). 

\subsection{Physical model}\label{sec2.2} \hfill\\
In the heterogeneous domain $\Omega\setminus\partial\omega_\#$, which consist of the 
particle volumes $\omega_\#$ and the porous space $\Omega\setminus\omega_\#$, we consider the electrostatic potential 
$\phi$ and $(n+1)$ components of concentrations of charged particles 
${c}=({c}_0,\dots,{c}_n)^\top$, $n\ge1$. 
The physical consistency requires positive concentrations ${c}>0$. 

At the external boundary $\partial\Omega$, we shall impose 
Dirichlet boundary conditions 
$\phi=\phi^{\rm bath}$ and ${c}={c}^{\rm bath}$ 
corresponding to a surrounding bath
and given by constant values $\phi^{\rm bath}\in\mathbb{R}$ and 
${c}^{\rm bath}=({c}^{\rm bath}_0,\dots,{c}^{\rm bath}_n)^\top 
\in\mathbb{R}^{n+1}_+$. 
We can then consider the normalised electrostatic potential  
$\phi-\phi^{\rm bath}$ and concentrations ${c}/{c}^{\rm bath}$ 
(i.e. ${c_s}/{c_s}^{\rm bath}$ for all $s=0,\dots,n$) and prescribe the
 following normalised Dirichlet conditions: 
\begin{equation}\label{2.1}
\phi=0,\qquad {c}=1\qquad\text{on $\partial\Omega$}. 
\end{equation}

In the following, all further relations will be formulated for the normalised 
potential and concentrations such that \eqref{2.1} holds. 

Let $z_s\in\mathbb{R}$ denote the electric charge of the $s$-th species 
with concentration ${c}_s$ for $s=0,\dots,n$. 
For the $n+1$- components of charges particles, we shall assume the following charge-neutrality 
\begin{equation}\label{2.2}
\sum_{s=0}^n z_s =0.
\end{equation}
A necessary condition for \eqref{2.2} is 
${\displaystyle\min_{s\in\{0,\dots,n\}}} z_s <0 
<{\displaystyle\max_{s\in\{0,\dots,n\}}} z_s $. 

The charge-neutrality assumption \eqref{2.2} 
implies also the following strong monotonicity property
\begin{equation}\label{2.3}
K |\xi|^2\le 
-\sum_{s=0}^n z_s\xi \exp({-z_s\xi})
\qquad\text{for all $\xi\in\mathbb{R}$}\qquad (K>0), 
\end{equation}
for a constant $K>0$, which 
follows
directly from Taylor expansion with respect to $(-z_s\xi)$.

We consider the following PNP steady-state system consisting of 
$(n+2)$ nonlinear, homogeneous equations: 
\begin{subequations}\label{2.4}
\begin{align}\label{2.4a}
-{\rm div} (\nabla {c}_s^\top D_s)&=0,
\qquad s=0,\dots,n,&&\quad\text{in }\omega_\#,\\
\label{2.4b}
-{\rm div} \bigl( (\nabla {c}_s 
+{\textstyle\frac{z_s}{{\kappa}T}} {c}_s \nabla\phi)^\top D_s\bigr) &=0,
\qquad s=0,\dots,n,&&\quad\text{in }\Omega\setminus\omega_\#,
\end{align}
\end{subequations}
\begin{subequations}\label{2.5}
\begin{align}\label{2.5a}
-{\rm div} (\nabla\phi^\top A^\eps) &=0, &&\text{in } \omega_\#,\\
\label{2.5b}
-{\rm div} (\nabla\phi^\top A^\eps) -\sum_{s=0}^n z_s {c}_s& =0, &&
\text{in } \Omega\setminus\omega_\#.
\end{align}
\end{subequations}
In both equations \eqref{2.4}, $D_s\in L^\infty(\Omega)^{d\times d}$, $D_s>0$, $s=0,\dots,n$ denote symmetric and positive definite diffusion matrices, which are in general discontinuous over $\partial\omega_\#$.
In \eqref{2.4b}, $\kappa>0$ is the Boltzmann constant, and $T>0$ is the temperature. We remark that the form of \eqref{2.4b} is based on assuming the Einstein relations for the mobilities. Moreover, eq. \eqref{2.4a} models the effect of charges particles being included into the solid particles, which is well known, for instance, for $Li^+$-ions, see e.g. \cite{Pic/11}. 
 
In \eqref{2.5}, $A\in L^\infty(\Upsilon)^{d\times d}$ denotes the symmetric and positive definite matrix of the electric permittivity, which 
oscillates periodically 
over cells according to $A^\eps(x) :=A(\{\frac{x}{\eps}\})$
and satisfies
\begin{equation}\label{2.6}
\begin{split}
&A^\top(y)=A(y),\qquad y\in\Upsilon\\ 
&\underline{K} |\xi|^2\le \xi^\top A(y)\xi \le\overline{K} |\xi|^2
\qquad\forall\xi\in\mathbb{R}^d, y\in\Upsilon, 
\qquad (0<\underline{K}<\overline{K}). 
\end{split}
\end{equation}
The entries of
the permittivity matrix $A$ are discontinuous functions in the cell $\Upsilon$ 
across the interface $\partial\omega$. 
A typical example
considers piecewise constant $A=\sigma_\omega I$ in $\omega$ and $A=\sigma_\Upsilon I$ in $\Upsilon\setminus\omega$, with material parameters $\sigma_\omega>0$ and $\sigma_\Upsilon>0$, where $I$ denotes here the identity matrix in $\mathbb{R}^{d\times d}$. 
In the following, we denote by $A_{ij}$, $i,j=1,\dots,d$, the matrix entries of $A$. 

From a physical point of view, \eqref{2.5a} represents Ohm's law in the solid phase. Moreover, we remark that the equations on $\omega_\#$, i.e. 
\eqref{2.4a} for ${c}$ and \eqref{2.5a} for $\phi$ are linear while the equations 
\eqref{2.4b} and \eqref{2.5b} on $\Omega\setminus\omega_\#$ form a coupled nonlinear problem  on the porous space. 
\medskip

The modelling of boundary conditions at the interfaces is a delicate issue. 
For the charge carries fluxes in \eqref{2.4}, we assume homogeneous Neumann conditions 
\begin{subequations}\label{2.7}
\begin{equation}\label{2.7a}
(\nabla {c}_s^-)^\top D_s\nu =0,
\qquad s=0,\dots,n,\quad\text{on}\quad \partial\omega_\#^-, 
\end{equation}
\begin{equation}\label{2.7b}
(\nabla {c}_s^+ +{\textstyle\frac{z_s}{{\kappa}T}} {c}_s^+ 
\nabla\phi^+)^\top D_s\nu =0,
\qquad s=0,\dots,n,\quad\text{on}\quad \partial\omega_\#^+.
\end{equation}
\end{subequations}
For the electrostatic potential in \eqref{2.5}, we suppose the Gouy--Chapman--Stern model for an Electric Double Layer (EDL) by assuming
the following inhomogeneous transmission boundary conditions 
(see \cite{Pic/11}): 
\begin{subequations}\label{2.8}
\begin{align}\label{2.8a}
(\nabla\phi^\top A^\eps)^- \nu -{\textstyle\frac{\alpha}{\eps}} 
[\![\phi]\!] &=\eps g,&&\text{on $\partial\omega_\#^-$},\\
\label{2.8b}
-(\nabla\phi^\top A^\eps)^+ \nu +{\textstyle\frac{\alpha}{\eps}} 
[\![\phi]\!] &=0, &&\text{on $\partial\omega_\#^+$}.
\end{align}
\end{subequations}
Here $\alpha\in\mathbb{R}_+$ and $g\in\mathbb{R}$ are material parameters 
given at the interface. 
We note that by summing \eqref{2.8a} and \eqref{2.8b}, we derive the relation 
\begin{equation}\label{2.9}
-[\![\nabla\phi^\top A^\eps]\!] \nu =\eps g,
\qquad\qquad\text{on $\partial\omega_\#$},
\end{equation}
implying that not only the electric potential $\phi$ but also fluxes $\nabla\phi^\top A^\eps\nu$ are discontinuous functions with jumps across the interface $\partial\omega_\#$. 

The asymptotic weights ${\textstyle\frac{1}{\eps}}$ in front of $[\![\phi]\!]$ 
and $\eps g$ at the right hand side of \eqref{2.8}, which were already mentioned in the introduction, 
shall be discussed in detail during the below asymptotic analysis as $\eps\searrow0^+$. 

We emphasise that the transmission conditions \eqref{2.8} couple 
the porous phase $\Omega\setminus\omega_\#$ with the solid phase $\omega_\#$ 
by means of the jump in $[\![\phi]\!]$. 
In fact, the transmission conditions \eqref{2.8} can be compared with the following two cases of simplified boundary conditions:
First, if $\phi$ were continuous across $\partial\omega_\#$, i.e. 
$[\![\phi]\!]=0$, then \eqref{2.8a} and \eqref{2.8b} would be decoupled 
into two usual Neumann boundary condition which do not represent the EDL. 
Second, if $\phi^-$ were known on the solid phase boundary 
$\partial\omega_\#^-$, then the model would reduced to a model on the porous 
space 
$\Omega\setminus\omega_\#$ with the following inhomogeneous Robin (Fourier) boundary condition (see \cite{FK})
\begin{equation*}
-(\nabla\phi^\top A^\eps)^+ \nu +{\textstyle\frac{\alpha}{\eps}} 
\phi^+ ={\textstyle\frac{\alpha}{\eps}} \phi^-, \qquad\text{on $\partial\omega_\#^+$}.
\end{equation*}
However, the subsequent homogenisation of this alternative model on the porous 
space 
$\Omega\setminus\omega_\#$ would nevertheless require a suitable continuation of 
$\phi^+$ onto $\omega_\#$. 

\subsection{Mathematical model}\label{sec2.3} \hfill\\
In the following, we shall amend the state variables with the superscript $\eps$ 
in order to highlight the dependency on the cell size.  

The physical model will be described
by the following weak variational formulation of the boundary value problem 
\eqref{2.1}, \eqref{2.4}--\eqref{2.5}, \eqref{2.7}--\eqref{2.8}: 
Find an electrostatic potential $\phi^\eps\in H^1(\Omega\setminus\partial\omega_\#)$ and $n+1$ components of charge carrier concentrations 
${c}^\eps\in H^1(\Omega\setminus\partial\omega_\#)^{n+1}\cap 
L^\infty(\Omega\setminus\partial\omega_\#)^{n+1}$ 
such that the concentrations are positive ${c}^\eps>0$ and satisfy 
\begin{equation}\label{2.10}
\phi^\eps =0,\qquad {c}^\eps =1\quad\text{on $\partial\Omega$},
\end{equation}
\begin{multline}\label{2.11}
\int_{\Omega\setminus\partial\omega_\#} 
\bigl(\nabla {c}^\eps_s +\chi_{{}_{\Omega\setminus\omega_\#}} 
{\textstyle\frac{z_s}{{\kappa}T}}\, {c}^\eps_s\, \nabla\phi^\eps 
\bigr)^\top D_s \nabla {c}_s \,dx =0,\qquad s=0,\dots,n,\\
\text{for all test-functions ${c}\in H^1(\Omega\setminus\partial\omega_\#)^{n+1}$: 
${c}=0$ on $\partial\Omega$},
\end{multline}
\begin{multline}\label{2.12}
\int_{\Omega\setminus\partial\omega_\#} \!\!\! \bigl(  
(\nabla\phi^\eps)^\top A^\eps\nabla\phi 
-\chi_{{}_{\Omega\setminus\omega_\#}}
\sum_{s=0}^n z_s {c}^\eps_s \phi \bigr)\,dx
+\int_{\partial\omega_\#} \!\!\! {\textstyle\frac{\alpha}{\eps}} 
[\![\phi^\eps]\!] [\![\phi]\!] \,dS_x\\
=\int_{\partial\omega_\#^-} \eps g \phi^- \,dS_x\qquad\
\text{for all $\phi\in H^1(\Omega\setminus\partial\omega_\#)$: 
$\phi=0$ on $\partial\Omega$}.
\end{multline}
Here $\chi_{{}_{\Omega\setminus\omega_\#}}$ denotes the characteristic function of the set $\Omega\setminus\omega_\#$. 

\begin{proposition}\label{prop2.1}
For strong solutions $(\phi^\eps,{c}^\eps)$,  
the variational system \eqref{2.10}--\eqref{2.12} and 
the boundary value problem \eqref{2.1}, \eqref{2.4}--\eqref{2.5}, \eqref{2.7}--\eqref{2.8} are equivalent. 
\end{proposition}

\begin{proof}
The assertion can be verified by usual variational arguments, which we briefly sketch for the sake of the reader.

The variational equations \eqref{2.11} and \eqref{2.12} are derived 
by multiplying the equations \eqref{2.4}--\eqref{2.5} with test-functions 
and subsequent integration by parts over
$\Omega\setminus\omega_\#$ and 
$\omega_\#$ due to boundary conditions \eqref{2.1} and \eqref{2.7}--\eqref{2.8}.

In return, given strong solutions $(\phi^\eps,{c}^\eps)$, the 
boundary
value problem \eqref{2.4}--\eqref{2.5}, \eqref{2.7}--\eqref{2.8} is obtained 
by varying the test-functions $(\phi,{c})$ in \eqref{2.11}, \eqref{2.12} 
and with the help of the following Green's 
formulas: By recalling the $\nu$ denotes both the outer normal on $\partial\Omega$ and 
$\partial\omega$, we have 
for all $p\in L^2_{\rm div}(\Omega\setminus\partial\omega_\#)^d$
\begin{subequations}\label{13}
\begin{multline}\label{13a}
\int_{\Omega\setminus\omega_\#} \!\!\! p^\top \nabla v\,dx
=-\int_{\Omega\setminus\omega_\#} \!\!\! v\,{\rm div}(p)\,dx
-\int_{\partial\omega_\#^+} p^\top v\nu \,dS_x\\
+\int_{\partial\Omega} p^\top v\nu \,dS_x,\qquad
\forall\ v\in H^1(\Omega\setminus\omega_\#), 
\end{multline}
\begin{equation}\label{13b}
\int_{\omega_\#} \!\!\! p^\top \nabla v\,dx
=-\int_{\omega_\#} \!\!\! v\,{\rm div}(p)\,dx
+\int_{\partial\omega_\#^-} p^\top v\nu \,dS_x,\qquad
\forall v\in H^1(\omega_\#),
\end{equation}
\end{subequations}
which are valid on $\Omega\setminus\omega_\#$ and $\omega_\#$, respectively. 
Hence, by suming \eqref{13a} and \eqref{13b}, we obtain the Green's formula 
representation 
\begin{equation}\label{green}
\int_{\Omega\setminus\partial\omega_\#} \!\!\! p^\top \nabla v\,dx
=-\int_{\Omega\setminus\partial\omega_\#} \!\!\! v\,{\rm div}(p)\,dx
-\int_{\partial\omega_\#} [\![p^\top v]\!]\nu \,dS_x
+\int_{\partial\Omega} p^\top v\nu \,dS_x,
\end{equation}
which holds on the disjoint domain $\Omega\setminus\partial\omega_\#$ 
for all $p\in L^2_{\rm div}(\Omega\setminus\partial\omega_\#)^d$ and  
$v\in H^1(\Omega\setminus\partial\omega_\#)$, 
see e.g. \cite{KK/00}. 
\end{proof}

The following Proposition \ref{prop2.2} states the crucial 
observation that introducing Boltzmann statistics 
allows to decouple the system of the homogeneous equations \eqref{2.11} 
and derive an equivalent scalar semi-linear Poisson-Boltzmann (PB) equation. 

\begin{proposition}\label{prop2.2}
The system \eqref{2.10}--\eqref{2.12} it is equivalent to 
the following nonlinear Poisson-Boltzmann equation: 
Find $\phi^\eps\in H^1(\Omega\setminus\partial\omega_\#)$ such that 
\begin{subequations}\label{2.13}
\begin{equation}\label{2.13a}
\phi^\eps =0\quad\text{on $\partial\Omega$},\\
\end{equation}
\begin{multline}\label{2.13b}
\int_{\Omega\setminus\partial\omega_\#} \bigl(  
(\nabla\phi^\eps)^\top A^\eps\nabla\phi 
-\sum_{s=0}^n z_s e^{- 
{\textstyle\frac{z_s}{{\kappa}T}}\, \chi_{{}_{\Omega\setminus\omega_\#}}\phi^\eps} \phi \bigr)\,dx\\
+\int_{\partial\omega_\#} \!\!\! {\textstyle\frac{\alpha}{\eps}} 
[\![\phi^\eps]\!] [\![\phi]\!] \,dS_x 
=\int_{\partial\omega_\#^-} \eps g \phi^- \,dS_x\\
\text{for all test-functions $\phi\in H^1(\Omega\setminus\partial\omega_\#)$: 
$\phi=0$ on $\partial\Omega$},
\end{multline}
\end{subequations}
together with the Boltzmann statistics determining ${c}^\eps$ from 
$\phi^\eps$, i.e.
\begin{equation}\label{2.14}
\begin{split}
{c}^\eps_s &=\exp\bigl( -{\textstyle\frac{z_s}{{\kappa}T}} \phi^\eps \bigr),
\qquad s=0,\dots,n,
\quad\text{a.e. on } \Omega\setminus\omega_\#,\\
{c}^\eps_s& \in\mathbb{R}_+,\qquad\qquad\qquad\, s=0,\dots,n,\quad \text{in } \omega_\#.
\end{split}
\end{equation}
\end{proposition}

\begin{proof}
Starting with \eqref{2.10}--\eqref{2.12}, we shall first prove the Boltzmann statistics \eqref{2.14} by introducing the entropy variables (the chemical potentials) 
\begin{equation}\label{2.15}
\begin{split}
&\mu^\eps_s := \ln {c}^\eps_s,\qquad s=0,\dots,n.
\end{split}
\end{equation}
Then, eq. \eqref{2.11} can be rewritten in terms of \eqref{2.15} in divergence form as 
\begin{multline}\label{2.16}
\int_{\Omega\setminus\partial\omega_\#} {c}^\eps_s 
\nabla \bigl( \mu^\eps_s +\chi_{{}_{\Omega\setminus\omega_\#}} 
{\textstyle\frac{z_s}{{\kappa}T}} \phi^\eps 
\bigr)^\top D_s \nabla {c}_s \,dx =0,\qquad s=0,\dots,n,\\
\text{for all test-functions ${c}\in H^1(\Omega\setminus\partial\omega_\#)^{n+1}$: 
${c}=0$ on $\partial\Omega$}.
\end{multline}
Due to the boundary condition \eqref{2.10}, we have 
$\phi^\eps =0=\mu^\eps$ on $\partial\Omega$ and
the test-function ${c}_s =\mu^\eps_s +\chi_{{}_{\Omega\setminus\omega_\#}} 
{\textstyle\frac{z_s}{{\kappa}T}} \phi^\eps$ can be inserted into \eqref{2.16}. 
Hence, by recalling that $D_s$ are symmetric and positive definite matrices 
and ${c}^\eps>0$, we derive the identity
$
\nabla\bigl(\mu^\eps_s +\chi_{{}_{\Omega\setminus\omega_\#}} 
{\textstyle\frac{z_s}{{\kappa}T}} \phi^\eps\bigr) =0,
$ $s=0,\dots,n,$ a.e. in $\Omega\setminus\partial\omega_\#$. 
Using again the boundary condition \eqref{2.10}, we conclude
\begin{equation}\label{2.17}
\mu^\eps_s +\chi_{{}_{\Omega\setminus\omega_\#}} 
{\textstyle\frac{z_s}{{\kappa}T}} \phi^\eps =0,\qquad s=0,\dots,n,\quad
\text{a.e. in } \Omega\setminus\omega_\#, 
\end{equation}
 and $\mu^\eps_s$ is an arbitrary constant in $\omega_\#$. 
This fact together with \eqref{2.15} implies \eqref{2.14}. 
By substituting the expressions \eqref{2.14} into equation \eqref{2.12} 
and by using the charge-neutrality \eqref{2.2} on $\omega_\#$, equation \eqref{2.13b} follows directly. 

Conversely, the equations \eqref{2.10}--\eqref{2.12} follow evidently from \eqref{2.13} and \eqref{2.14}. This completes the proof. 
\end{proof}

We remark that the concentrations ${c}^\eps$ in \eqref{2.14} are unique 
up to fixing the constant positive values within the solid particles $\omega_\#$. 
\medskip

By exploiting Proposition~\ref{prop2.2}, we construct a solution $(\phi^\eps,{c}^\eps)$ 
for the variational problem \eqref{2.10}--\eqref{2.12} from the 
scalar problem \eqref{2.13} for the potential $\phi^\eps$. 
The $n+1$ concentrations ${c}^\eps$ are afterwards explicitly determined by \eqref{2.14}. 

\begin{theorem}\label{theo2.1}
There exists the unique solution $\phi^\eps$ to the semilinear problem 
\eqref{2.13} satisfying the following residual estimate
\begin{equation}\label{2.18}
\|\nabla \phi^\eps\|_{L^2(\Omega\setminus\partial\omega_\#)}^2 
+{\textstyle\frac{1}{\eps}
\|[\![\phi^\eps]\!]\|_{L^2(\partial\omega_\#)}^2 
+\|\phi^\eps\|_{L^2(\Omega\setminus\omega_\#)}^2}
={\rm O} (1),
\end{equation}
which is uniform with respect to $\eps>0$. 
\end{theorem}

\begin{proof}
We first emphasise that 
for the first two terms on the left hand side of \eqref{2.18}
the following discontinuous version of Poincare's inequality for homogeneous Dirichlet condition \eqref{2.13a} 
holds on 
the multiple domains $\Omega\setminus\partial\omega_\#$ without interfaces $\partial\omega_\#$ (see e.g. \cite{Hum/00,Orl/12}):
\begin{equation}\label{2.19}
K_0 \|\phi^\eps\|_{H^1(\Omega\setminus\partial\omega_\#)}^2 
\le\|\nabla \phi^\eps\|_{L^2(\Omega\setminus\partial\omega_\#)}^2 
+{\textstyle\frac{1}{\eps}} 
\|[\![\phi^\eps]\!]\|_{L^2(\partial\omega_\#)}^2, 
\quad (K_0>0).
\end{equation}
Therefore, the lower estimate \eqref{2.19} together with  \eqref{2.3} ensures 
the coercivity of the operator of the problem \eqref{2.13b}. 

The main difficulty of the existence proof arises from the unbounded, exponential growth of the nonlinear term in \eqref{2.13b}.
While classic existence theorems on quasilinear equations are thus not applicable here,  
the solution can nevertheless be constructed by a thresholding, 
see e.g. \cite{LC/06} and references therein for the details. 

To derive the estimate \eqref{2.18}, it suffices to insert $\phi=\phi^\eps$ as the test-function 
in the variational equation \eqref{2.13b} and apply \eqref{2.3} in order to estimate below the nonlinear term 
at the left hand side 
of \eqref{2.13b}. 
Finally the right hand side of \eqref{2.13b} can be estimated by means of the following trace theorem 
\begin{equation}\label{2.20}
\begin{split}
&\int_{\partial\omega_\#^-} \eps g \phi^- \,dS_x
\le |g| \|\phi\|_{H^1(\Omega\setminus\partial\omega_\#)}, 
\end{split}
\end{equation}
see \cite{BPC/01} for the details. 
This completes the proof. 
\end{proof}

We remark that in the following Section \ref{sec3}, we will refine the residual error estimate \eqref{2.18} 
by means of asymptotic analysis as $\eps\searrow0^+$ and homogenisation. 

\section{Homogenisation and residual error estimate}\label{sec3}

We start the homogenisation procedure with three auxiliary cell problems. 
The first two cell problems serve to expand the inhomogeneous boundary traction $g$ 
and the volume potential of the variational problem \eqref{2.13} from the porous space 
$\Omega\setminus\omega_\#$ onto the whole domain $\Omega\setminus\partial\omega_\#$. 

The third cell problem is needed to decompose the matrix $A^\eps$ of 
oscillating coefficients in the cells with respect to small $\eps\searrow0^+$. 
This procedure will result in a regular asymptotic decomposition of 
the perturbation problem with a subsequent error estimate of the corrector term. 

For a generic cell $\Upsilon$, we introduce the Sobolev space $H^1_\#(\Upsilon)$ 
of functions which can be extended periodically to $H^1(\mathbb{R}^d)$. 
This requires matching traces on the opposite faces of $\partial\Upsilon$. 
Moreover, we shall denote by $H^1_\#(\Upsilon\setminus\partial\omega)$ those periodic functions, 
which are discontinuous, i.e. allow jumps across the interface $\partial\omega$. 

\subsection{Auxiliary results}\label{sec3.1}\hfill\\
We state the first auxiliary cell problem as follows: 
Find $L\in H^1(\Upsilon\setminus\partial\omega)$ such that 
\begin{multline}\label{3.1}
\int_{\Upsilon\setminus\partial\omega} 
(\nabla L^\top A\nabla u +L u) \,dy
=\int_{\partial\omega^-} u^- \,dS_y\\
\text{for all test-functions } u\in H^1(\Upsilon\setminus\partial\omega).
\qquad
\end{multline}
In view of the homogenisation result stated in Theorem 
\ref{theo3.1} in Section~\ref{sec3.2} below, the auxiliary problem \eqref{3.1} serves to expand the inhomogeneity of the boundary condition 
\eqref{2.8a} given by the material parameter $g$ 
in terms of the weak formulation stated in \eqref{2.13b}. 

The existence of a unique solution $L$ in \eqref{3.1} follows via standard elliptic theory from the assumed properties \eqref{2.6} of $A$.
With its help, we are able to prove the following result. 

\begin{lemma}[The cell boundary-traction problem]\label{lem3.1}\hfill\\
For all test-fucntions $\phi\in H^1(\Omega\setminus\partial\omega_\#)$: 
$\phi=0$ on $\partial\Omega$ holds the following expansion
\begin{equation}\label{3.2}
\int_{\partial\omega_\#^-} \eps g \phi^- \,dS_x
-\int_{\Omega\setminus\partial\omega_\#} 
{\textstyle\frac{|\partial\omega|}{|\Upsilon|}} g\phi\,dx
=\eps\, l_1(\phi),
\end{equation}
where $l_1: H^1(\Omega\setminus\partial\omega_\#)\mapsto\mathbb{R}$ 
is a linear form satisfying
\begin{equation}\label{3.3}
|l_1(\phi)|\le K\|\phi\|_{H^1(\Omega\setminus\partial\omega_\#)}, \qquad (K>0). 
\end{equation}
\end{lemma}

\begin{proof}
We apply the auxiliary cell problem \eqref{3.1}. 
By inserting a constant test-function $u$, we calculate the average value 
\begin{equation}\label{3.4}
\langle L\rangle_y ={\textstyle\frac{|\partial\omega|}{|\Upsilon|}},
\qquad\text{where}\quad
\langle L\rangle_y :={\textstyle\frac{1}{|\Upsilon|}}
\int_{\Upsilon\setminus\partial\omega} L \,dy.
\end{equation}
Here, $|\partial\omega|$ and $|\Upsilon|$ denote the Hausdorff measures of 
the solid particle boundary $\partial\omega$ in $\mathbb{R}^{d-1}$ and 
of the cell $\Upsilon$ in $\mathbb{R}^d$, respectively. 

Subtracting 
$\int_{\Upsilon\setminus\partial\omega} 
\langle L\rangle_y u \,dy$ 
from \eqref{3.1}, we rewrite it equivalently as 
\begin{multline}\label{3.5}
\int_{\partial\omega^-} u^- \,dS_y 
-\int_{\Upsilon\setminus\partial\omega} 
 \langle L\rangle_y u \,dy\\ 
=\int_{\Upsilon\setminus\partial\omega} 
\bigl(\nabla_{\!y} L^\top A\nabla_{\!y} u +(L -\langle L\rangle_y) 
(u-\langle u\rangle_y) \bigr) \,dy =:l(u),
\end{multline}
where we have added to the residuum $l(u)$ the trivial term 
\[
\int_{\Upsilon\setminus\partial\omega} 
(L -\langle L\rangle_y) \langle u\rangle_y \,dy=0,\qquad
\langle u\rangle_y :={\textstyle\frac{1}{|\Upsilon|}}
\int_{\Upsilon\setminus\partial\omega} u \,dy.
\]
In the following, we shall apply the discontinuous Poincare inequality 
\begin{equation}\label{26a}
K_1
\|u-\langle u\rangle_y\|_{L^2(\Upsilon\setminus\partial\omega)}
\le \|\nabla_{\!y} u\|_{L^2(\Upsilon\setminus\partial\omega)}
+\|[\![u]\!]\|_{L^2(\partial\omega)},\qquad 
(K_1>0),
\end{equation}
and the Trace Theorem
\begin{equation}\label{26aa}
\|[\![u]\!]\|_{L^2(\partial\omega)}
\le
{\textstyle\frac{K_2}{\sqrt{2}}} \bigl( \|\nabla_{\!y} u \|_{L^2(\Upsilon\setminus\partial\omega)}
+\|u \|_{L^2(\Upsilon\setminus\partial\omega)} \bigr)
\le K_2 \|u \|_{H^1(\Upsilon\setminus\partial\omega)},
\end{equation}
with $K_2>0$,
which combine to the estimate 
\begin{equation}\label{26aaa}
\|u-\langle u\rangle_y\|_{L^2(\Upsilon\setminus\partial\omega)}
\le K_3\|u \|_{H^1(\Upsilon\setminus\partial\omega)},\qquad 
(K_3 =K_1^{-1}(1+K_2)).
\end{equation}

By recalling that $A\in L^\infty(\Upsilon)^{d\times d}$ and by applying Cauchy's inequality to the right hand side of \eqref{3.5} and subsequently applying estimate \eqref{26aaa} to $L$ and $u$, we obtain
the following estimate 
\begin{align}\label{26b}
|l(u)| \le&\ \overline{K} \|\nabla L\|_{L^2(\Upsilon\setminus\partial\omega)} 
\|\nabla u\|_{L^2(\Upsilon\setminus\partial\omega)}\nonumber
+K_3^2 \|L \|_{H^1(\Upsilon\setminus\partial\omega)} \|u \|_{H^1(\Upsilon\setminus\partial\omega)}\nonumber\\
\le&\ (\overline{K}+K_3^2) \|L \|_{H^1(\Upsilon\setminus\partial\omega)} \|u \|_{H^1(\Upsilon\setminus\partial\omega)}
\end{align}
with $\overline{K}$ from \eqref{2.6} and $K_3$ from \eqref{26aaa}.

For a proper test-function $\phi(x)$ with 
$x=\eps\bigl\lfloor\frac{x}{\eps}\bigr\rfloor +\eps\{\frac{x}{\eps}\}$, we insert 
$u(x,y)=\phi(\eps\bigl\lfloor\frac{x}{\eps}\bigr\rfloor +\eps y)$ into \eqref{3.5} 
and apply the periodic coordinate transformation 
$y\mapsto x$, $\Upsilon\mapsto\mathbb{R}^d$, 
by paving $\mathbb{R}^d$ such that $\{\frac{x}{\eps}\} =y$ 
(recall Section~\ref{sec2.1}). 
After observing that $dy\mapsto \eps^{-d} dx$, $dS_y\mapsto \eps^{1-d} dS_x$, 
$\nabla_y\mapsto \eps\nabla_x$, we also multiply \eqref{3.5} with the constant $g\eps^{d}$ and use \eqref{3.4} in order to derive 
\begin{equation*}
\sum_{p=1}^{N_\eps}  \int_{(\partial\omega^\eps_p)^-} \eps g \phi^- \,dS_x 
-\sum_{p=1}^{N_\eps} 
\int_{\Upsilon^\eps_p\setminus\partial\omega^\eps_p}  
{\textstyle\frac{|\partial\omega|}{|\Upsilon|}} g \phi \,dx\\ 
=\eps\, l_1(\phi),
\end{equation*}
which is \eqref{3.2} with 
the following right hand side term:
\begin{equation}\label{26c}
l_1(\phi) :=g \sum_{p=1}^{N_\eps} 
\int_{\Upsilon^\eps_p\setminus\partial\omega^\eps_p} 
\bigl( (\eps\nabla_{\!x} L^\eps)^\top A^\eps\nabla_{\!x} \phi 
+(L^\eps -\langle L\rangle_y) \cdot {\textstyle\frac{1}{\eps}} 
(\phi-\langle \phi\rangle_y) \bigr) \,dx,
\end{equation}
where we denote $L^\eps(x) :=L(\{\frac{x}{\eps}\})$ and 
$A^\eps(x) :=A(\{\frac{x}{\eps}\})$. 

Similarly, the discontinuous Poincare inequality \eqref{26a}
and the trace theorem \eqref{26aa} transform, respectively, into 
\begin{subequations}\label{26aaaa}
\begin{equation}\label{26daaaaa}
\begin{split}
\frac{K_1}{\eps}
 \|\phi-\langle \phi\rangle_y \|_{L^2(\Upsilon^\eps_p\setminus\partial\omega^\eps_p)} 
\le \| \nabla_{\!x} \phi
\|_{L^2(\Upsilon^\eps_p\setminus\partial\omega^\eps_p)}
+\frac{1}{\sqrt{\eps}} \|[\![\phi]\!]\|_{L^2(\partial\omega^\eps_p)}, 
\end{split}
\end{equation}
\begin{equation}\label{26aaaab}
\begin{split}
{\frac{1}{\sqrt{\eps}}} \|[\![\phi]\!] 
\|_{L^2(\partial\omega^\eps_p)} 
&\le {\frac{K_2}{\sqrt{2}}}  \Bigl(
\|\nabla_{\!x} \phi \|_{L^2(\Upsilon^\eps_p\setminus\partial\omega^\eps_p)}
+{\frac{1}{\eps}} \| \phi 
\|_{L^2(\Upsilon^\eps_p\setminus\partial\omega^\eps_p)} \Bigr)\\
& \le K_2 \|\phi \|_{H^1(\Upsilon^\eps_p\setminus\partial\omega^\eps_p)}, 
\qquad p=1,\dots,N_\eps,
\end{split}
\end{equation}
\end{subequations}
which combines to the uniform estimate 
\begin{equation}\label{26d}
\frac{1}{\eps} 
\|\phi-\langle \phi\rangle_y\|_{L^2(\Upsilon^\eps_p\setminus\partial\omega^\eps_p)}
\le K_3\|\phi \|_{H^1(\Upsilon^\eps_p\setminus\partial\omega^\eps_p)}
\end{equation}
with $K_3>0$ from \eqref{26aaa}. 
We note that the first line of \eqref{26aaaab} expresses the $H^1$-norm 
by the standard homogeneity argument, see e.g. \cite[Appendix, Lemma~1, p.370]{SP/80}.

Therefore, the estimate 
\eqref{26b}
of $l$ yields the following estimate of $l_1$ 
\begin{align}
|l_1(\phi)|&\le |g| \sum_{p=1}^{N_\eps} \Bigl(
\overline{K} \|\nabla_y L\|_{L^2(\Upsilon\setminus\partial\omega)} 
\|\nabla_x \phi\|_{L^2(\Upsilon^\eps_p\setminus\partial\omega^\eps_p)}\nonumber \\
&\quad +\|L -\langle L\rangle_y\|_{L^2(\Upsilon\setminus\partial\omega)}
\cdot{\frac{1}{\eps}} \| \phi -\langle \phi\rangle_y
\|_{L^2(\Upsilon^\eps_p\setminus\partial\omega^\eps_p)}\Bigr)\nonumber \\
&\le |g| (\overline{K}+K_3^2) \|L \|_{H^1(\Upsilon\setminus\partial\omega)} 
\sum_{p=1}^{N_\eps}\|\phi \|_{H^1(\Upsilon^\eps_p\setminus\partial\omega^\eps_p)}. \label{27}
\end{align}
Here we used \eqref{2.6} and inequalities \eqref{26aaa} for $L$ and \eqref{26d} for $\phi$.  
Then, \eqref{27} follows \eqref{3.3} with the constant 
$K =|g| (\overline{K}+K_3^2) \|L \|_{H^1(\Upsilon\setminus\partial\omega)}$, 
which completes the proof. 
\end{proof}

\begin{remark}
We remark that Lemma~\ref{lem3.1} justifies not only the a-priori estimate \eqref{2.20}, but also refines it
by specifying the limiting asymptotic term as $\eps\searrow0^+$, which consists of the constant potential 
${\textstyle\frac{|\partial\omega|}{|\Upsilon|}} g$ 
distributed uniformly over $\Omega$. 
\end{remark}

The next auxiliary cell problem studies
the asymptotic expansion of a volume force $f\in H^1(\Omega\setminus\partial\omega_\#)$, which is given 
on the porous space $\Upsilon\setminus\omega$ 
surrounding the solid particle $\omega\subset\Upsilon$. 
It will be applied in particular to the nonlinear term in \eqref{2.13b}, i.e.
we shall consider the specific volume force 
$f(x)=-\sum_{s=0}^n z_s \exp({-{\textstyle\frac{z_s}{{\kappa}T}} \phi^0(x)})$
in Theorem~\ref{theo3.1} below.
\bigskip
 
With $x=\eps\lfloor\frac{x}{\eps}\rfloor +\eps\{\frac{x}{\eps}\}$ 
(recall Section~\ref{sec2.1}), the following unfolding operator 
$$
T_\eps:
\begin{cases}
H^1(\Omega\setminus\partial\omega_\#)\mapsto 
H^1\bigl((\Omega\setminus\partial\omega_\#)
\times(\Upsilon\setminus\partial\omega)\bigr),\\[1mm]
(T_\eps f)
(x,y) :=f(\eps\bigl\lfloor\frac{x}{\eps}\bigr\rfloor +\eps y),
\end{cases}
$$
is well defined, see \cite{CDDGZ/12}. 
For its modification near the boundaries $\partial\Omega$ of non-rectangular domains $\Omega$, see \cite{Fra/10}. 

For $x\in\Omega\setminus\partial\omega_\#$, there exists a function 
$M(x,y)$ piecewisely composed of solutions $M(x,\,\cdot\,)$ of the following $x$-dependent cell problems 
(compare with \eqref{3.1}): 
Find $M(x,\,\cdot\,)\in H^1(\Upsilon\setminus\partial\omega)$ such that 
\begin{multline}\label{3.6}
\int_{\Upsilon\setminus\partial\omega} 
(\nabla_{\!y} M^\top A\nabla_{\!y} u +M u) \,dy
=\int_{\Upsilon\setminus\omega} (T_\eps f)\, u \,dy\\
\text{for all test-functions $u\in H^1(\Upsilon\setminus\partial\omega)$}.
\end{multline}

\begin{lemma}[Unfolding of the cell volume-force problem]\label{lem3.2}\hfill\\
For all $\phi\in H^1(\Omega\setminus\partial\omega_\#)$: 
$\phi=0$ on $\partial\Omega$ holds the following expansion 
\begin{equation}\label{3.7}
\int_{\Omega\setminus\omega_\#} f\phi \,dx
-{\textstyle\frac{|\Upsilon\setminus\omega|}{|\Upsilon|}} 
\int_{\Omega\setminus\partial\omega_\#} f\phi\,dx 
=\eps\, l_2(\phi),
\end{equation}
where $l_2: H^1(\Omega\setminus\partial\omega_\#)\mapsto\mathbb{R}$ is a linear form satisfying
\begin{equation}\label{3.8}
|l_2(\phi)|\le K\|\phi\|_{H^1(\Omega\setminus\partial\omega_\#)}, \qquad (K>0). 
\end{equation}
\end{lemma}

\begin{proof} 
By inserting a constant test-function $u$ into the auxiliary cell problem \eqref{3.6}, we obtain 
the locally averaged value of $M=M(x,y)$ 
\begin{equation}\label{3.9}
\begin{split}
&\langle M(x,\,\cdot\,) {\rangle_y} 
:={\textstyle\frac{1}{|\Upsilon|}}
\int_{\Upsilon\setminus\partial\omega} M \,dy 
={\textstyle\frac{1}{|\Upsilon|}}
\int_{\Upsilon\setminus\omega} T_\eps f \,dy.
\end{split}
\end{equation}
Moreover, by using the average $\langle T_\eps f
{\rangle_y}$, we can expand 
\begin{equation}\label{3.10}
F(x,y) :=(T_\eps f)(x,y) -\langle T_\eps f{\rangle_y},\qquad 
\langle T_\eps f
{\rangle_y} :={\textstyle\frac{1}{|\Upsilon|}}
\int_{\Upsilon\setminus\partial\omega} (T_\eps f)(x,\cdot) \,dy.
\end{equation}
See \cite{KK/14} for the analysis of expansion \eqref{3.10} in terms of Fourier series. 
For fixed $x$ the residual $F(x,y)$ has zero average $\langle F\rangle_y
=0$ and estimates as
\begin{multline}\label{3.11}
\|F(x,\,\cdot\,)
\|_{L^2(\Upsilon\setminus\partial\omega)} 
=\|T_\eps f-\langle T_\eps f \rangle_y 
\|_{L^2(\Upsilon\setminus\partial\omega)}
\le
K_3 \|T_\eps f\|_{H^1(\Upsilon\setminus\partial\omega)}
\end{multline}
due to the discontinuous Poincare inequality \eqref{26aaa}. 
By inserting \eqref{3.10} into \eqref{3.9}, we calculate \begin{equation*}
\langle M
\rangle_y
={\textstyle\frac{1}{|\Upsilon|}}
\int_{\Upsilon\setminus\omega} T_\eps f \,dy=
{\textstyle\frac{1}{|\Upsilon|}}
\int_{\Upsilon\setminus\omega} F\,dy+
{\textstyle\frac{|\Upsilon\setminus\omega|}{|\Upsilon|}}\,
\langle T_\eps f
\rangle_y,
\end{equation*}
and thus derive by using again \eqref{3.10}, i.e. $ \langle T_\eps f \rangle_y=T_\eps f-F$
\begin{equation}\label{3.12}
{\textstyle\frac{|\Upsilon\setminus\omega|}{|\Upsilon|}}\, T_\eps f
=\langle M
\rangle_y 
+{\textstyle\frac{|\Upsilon\setminus\omega|}{|\Upsilon|}} 
\Bigl(F -{\textstyle\frac{1}{|\Upsilon\setminus\omega|}} 
\int_{\Upsilon\setminus\omega} F \,dy\Bigr).
\end{equation}
After multiplying the identity \eqref{3.12} with $u$ and integrating it over $\Upsilon\setminus\partial\omega$, we 
subtract it from \eqref{3.6} and rewrite \eqref{3.6} equivalently as 
\begin{multline}\label{3.13}
\int_{\Upsilon\setminus\omega} (T_\eps f)\, u \,dy
-{\textstyle\frac{|\Upsilon\setminus\omega|}{|\Upsilon|}}
\int_{\Upsilon\setminus\partial\omega} (T_\eps f)\, u \,dy\\
=-{\textstyle\frac{|\Upsilon\setminus\omega|}{|\Upsilon|}} 
\int_{\Upsilon\setminus\partial\omega} 
\biggl(F -{\textstyle\frac{1}{|\Upsilon\setminus\omega|}} 
\int_{\Upsilon\setminus\omega} F \,dy\biggr) u \,dy\qquad\qquad\qquad\qquad \\ 
+\int_{\Upsilon\setminus\partial\omega} 
\bigl(\nabla_{\!y} M^\top A\nabla_{\!y} u +(M -\langle M
\rangle_y) 
(u-\langle u
\rangle_y) \bigr) \,dy =:m(u),
\end{multline}
where we have added the trivial term 
$\int_{\Upsilon\setminus\partial\omega} 
(M -\langle M\rangle_y) \langle u\rangle_y \,dy=0$ and the residuum
$m(u)$ shortly denotes the right hand side terms of \eqref{3.13}.

For fixed $x\in\Omega\setminus\partial\omega_\#$, Cauchy's inequality yields for the first term on the right hand side of  \eqref{3.13}
\begin{multline}\label{3.13a}
\Bigl| \int_{\Upsilon\setminus\partial\omega} 
\Bigl(F -{\textstyle\frac{1}{|\Upsilon\setminus\omega|}} 
\int_{\Upsilon\setminus\omega} F \,dy\Bigr) u \,dy \Bigr|\\
\le \|F\|_{L^2(\Upsilon\setminus\partial\omega)} 
\|u\|_{L^2(\Upsilon\setminus\partial\omega)}
+\sqrt{\textstyle\frac{|\Upsilon|}{|\Upsilon\setminus\omega|}}\,
\|F\|_{L^2(\Upsilon\setminus\omega)} 
\|u\|_{L^2(\Upsilon\setminus\partial\omega)}\\
\le \Bigl(1 +\sqrt{\textstyle\frac{|\Upsilon|}{|\Upsilon\setminus\omega|}}\Bigr)
\|F\|_{L^2(\Upsilon\setminus\partial\omega)} 
\|u\|_{L^2(\Upsilon\setminus\partial\omega)}.
\end{multline}
Thus, by applying the estimates \eqref{3.11} 
and \eqref{3.13a}
to $F$ and the discontinuous Poincare inequality 
\eqref{26aaa}
to $M$ and $u$, 
we estimate $m(u)$ at the right hand side of \eqref{3.13} as   
\begin{multline}\label{3.14}
|m(u)|\le K_4 \|T_\eps f\|_{H^1(\Upsilon\setminus\partial\omega)}
\|u \|_{L^2(\Upsilon\setminus\partial\omega)}\\
+(\overline{K}+K_3^2) \|M \|_{H^1(\Upsilon\setminus\partial\omega)} 
\|u \|_{H^1(\Upsilon\setminus\partial\omega)},
\end{multline}
where $K_4 ={\textstyle\frac{|\Upsilon\setminus\omega|}{|\Upsilon|}} 
\Bigl(1 +\sqrt{\textstyle\frac{|\Upsilon|}{|\Upsilon\setminus\omega|}} 
\Bigr) K_3$ and by recalling $\overline{K}$ from \eqref{2.6} and $K_3$ from \eqref{26aaa}.

Next, we substitute $u=(T_\eps \phi)$ 
as the test-function in \eqref{3.13} 
and use the property 
$T_\eps f\cdot T_\eps \phi =T_\eps (f\phi)$ 
of the unfolding operator. 
After applying the periodic coordinate transformation 
$y\mapsto x$, $\{\frac{x}{\eps}\} =y$ to \eqref{3.13} similar to the proof of Lemma \ref{lem3.1}, we arrive with  
$T_\eps (f\phi)\mapsto f\phi$ and $T_\eps \phi\mapsto \phi$ 
at \eqref{3.7} with 
\begin{multline}\label{3.15}
l_2(\phi) :=\sum_{p=1}^{N_\eps} \Bigl[ {\textstyle\frac{1}{\eps}} 
{\textstyle\frac{|\Upsilon\setminus\omega|}{|\Upsilon|}}
\int_{\Upsilon^\eps_p\setminus\partial\omega^\eps_p} 
\Bigl(F^\eps -{\textstyle\frac{1}{|\Upsilon\setminus\omega|}} 
\int_{\Upsilon\setminus\omega} F \,dy\Bigr) \phi \,dx\\ 
+\int_{\Upsilon^\eps_p\setminus\partial\omega^\eps_p} 
\bigl( (\eps\nabla_{\!x} M^\eps)^\top A^\eps\nabla_{\!x} \phi 
+(M^\eps -\langle M\rangle_y) {\textstyle\frac{1}{\eps}} 
(\phi-\langle T_\eps \phi\rangle_y) \bigr) \,dx\Bigr],
\end{multline}
where $F^\eps(x) :=F(x,\{\frac{x}{\eps}\})$ and 
$M^\eps(x) :=M(x,\{\frac{x}{\eps}\})$. 
Similarly to \eqref{3.14}, we estimate with
$F^\eps(x)=f(x)-\langle T_\eps f\rangle_y(x)$
\begin{equation*}
\begin{split}
|l_2(\phi)|
\le&\ \sum_{p=1}^{N_\eps} \Bigl[ 
{\textstyle\frac{|\Upsilon\setminus\omega|}{|\Upsilon|}} 
\Bigl(1 \!+\! \sqrt{\textstyle\frac{|\Upsilon|}{|\Upsilon\setminus\omega|}} \Bigr)
{\textstyle\frac{1}{\eps}} 
\|f-\langle T_\eps f\rangle_y\|_{L^2(\Upsilon^\eps_p\setminus\partial\omega^\eps_p)} 
\|\phi\|_{L^2(\Upsilon^\eps_p\setminus\partial\omega^\eps_p)}\\
&+\sup_{x\in\Omega\setminus\partial\omega_\#} \!\! \Bigl\{
\overline{K} \|\nabla_{\!y} M(x,\,\cdot\,)\|_{L^2(\Upsilon\setminus\partial\omega)} 
\|\nabla \phi\|_{L^2(\Upsilon^\eps_p\setminus\partial\omega^\eps_p)}\\
&+\|M(x,\,\cdot\,) -\langle M(x,\,\cdot\,)\rangle_y\|_{L^2(\Upsilon\setminus\partial\omega)}
{\textstyle\frac{1}{\eps}} \| \phi -\langle T_\eps\phi\rangle_y
\|_{L^2(\Upsilon^\eps_p\setminus\partial\omega^\eps_p)}\Bigr\}\Bigr],\\
\end{split}
\end{equation*}
hence, 
\begin{equation}\label{37a}
\begin{split}
|l_2(\phi)| \le&\ K_4 \|f\|_{H^1(\Omega\setminus\partial\omega_\#)}
\|\phi\|_{L^2(\Omega\setminus\partial\omega_\#)}\\
&+(\overline{K}+K_3^2) \sup_{x\in\Omega\setminus\partial\omega_\#} 
\|M(x,\,\cdot\,)\|_{H^1(\Upsilon\setminus\partial\omega)}
\|\phi\|_{H^1(\Omega\setminus\partial\omega_\#)},
\end{split}
\end{equation}
where we have used \eqref{26aaa} for $M(x,\,\cdot\,)$ and \eqref{26d} for $f$ and $\phi$.
Thus, \eqref{37a} implies the estimate \eqref{3.8} of the residual term $l_2$ given in \eqref{3.15} 
with 
\[
K =K_4 \|f\|_{H^1(\Omega\setminus\partial\omega_\#)} 
+(\overline{K}+K_3^2)  \!\!\! \sup_{x\in\Omega\setminus\partial\omega_\#} \!\!\! 
\|M(x,\,\cdot\,)\|_{H^1(\Upsilon\setminus\partial\omega)}.
\] 
This completes the proof. 
\end{proof}

\begin{remark}
We remark that the factor 
${\textstyle\frac{|\Upsilon\setminus\omega|}{|\Upsilon|}}$ 
in \eqref{3.7} reflects the porosity of the cell $\Upsilon$ due to the 
presence of the solid particles $\omega$. 
In our particular geometric setting, we have $|\Upsilon|=1$ and 
$|\Upsilon\setminus\omega|=1-|\omega|$, respectively. 
\end{remark}

The third cell problem considers the solutions of the following system of $d$ linear equations: 
Find a vector of periodic functions $N=(N_1,\dots,N_d)^\top\in H^1_\#(\Upsilon\setminus\partial\omega)^d$ 
with componentwise zero average $\langle N
\rangle_y
=0$ such that 
\begin{multline}\label{3.16}
\int_{\Upsilon\setminus\partial\omega} D(N +y) A\nabla u \,dy
+\int_{\partial\omega} \alpha [\![N]\!] [\![u]\!] \,dS_y =0,\\
\text{for all scalar test-functions $u\in H^1_\#(\Upsilon\setminus\partial\omega)$}.
\end{multline}
Here, $H^1_\#(\Upsilon\setminus\partial\omega)$ denotes the space of periodic $H^1$-functions and $D N(y)\in\mathbb{R}^{d\times d}$ 
for $y\in\Upsilon\setminus\partial\omega$ stands for the row-wise gradient matrix of the vector $N$, that is
$$
D N :=
\begin{pmatrix}
N_{1,1} & \dots & N_{1,d}\\
\vdots & & \vdots\\
N_{d,1} & \dots & N_{d,d}
\end{pmatrix},
\qquad
\text{where}\quad N_{i,j} :={\textstyle\frac{\partial N_i}{\partial y_j}}, 
\quad i,j=1,\dots,d.
$$ 
Moreover in \eqref{3.16}, $D y = I\in\mathbb{R}^{d\times d}$ yields 
the identity matrix.  
The solvability of \eqref{3.16} follows from the symmetry and positive definiteness assumption \eqref{2.6}.
The uniqueness of the solution $N$ is provided due to the constraint 
$\langle N
\rangle_y
=0$. 
Indeed, since $N(y)+K$ with an arbitrary constant $K$ solves also \eqref{3.16}, 
the zero average condition is sufficient (and necessary) to ensure the uniqueness of the solution, see e.g. \cite{OSY/92}.
Finally, the solution is smooth locally in $\Upsilon\setminus\partial\omega$. 

\begin{remark}
We remark in particular, that if $[\![N]\!] =[\![u]\!] =0$ would hold, 
then the discontinuous cell problem \eqref{3.16} would reduce to a standard, continuous cell problem. 
\end{remark}

The system \eqref{3.16} is essential to determine the efficient coefficient 
matrix $A^0$ of the macroscopic model averaged over $\Omega$. 
In fact, following the lines of \cite{OSY/92,ZKO/94}, 
we shall establish an orthogonal decomposition of Helmholtz type for 
the oscillating coefficients $A^\eps$. 

The Helmholtz type decomposition is based on the left hand side of \eqref{3.16} defining an inner product 
$
\langle\!\langle\,\cdot\,,\,\cdot\,\rangle\!\rangle$ 
in $H^1_\#(\Upsilon\setminus\partial\omega)$. 
Due to $[\![y]\!] =0$, the variational equation \eqref{3.16} reads as
$\langle\!\langle N+y,u\rangle\!\rangle
=0$ for all $u\in H^1_\#(\Upsilon\setminus\partial\omega)$, 
which implies that $N+y$ belongs to the kernel of this topological vector space. 
Thus, the fundamental theorem of vector calculus (the Helmholtz theorem,  
see e.g. \cite{ZKO/94})
permits the following representation
as sum of a constant matrix $A^0$ and divergence free $B(y)$ fields in 
$\mathbb{R}^{d\times d}$:
\begin{equation}\label{3.23}
D \bigl( N(y)+y\bigr) A (y) =A^0 +B(y),
\qquad 
\text{a.e.}\quad
y\in \Upsilon\setminus\partial\omega,
\end{equation}
where $B$ has zero average, i.e. 
$$
0 =\langle B\rangle_y :={\textstyle\frac{1}{|\Upsilon|}}
\int_{\Upsilon\setminus\partial\omega} B(y) \,dy.
$$
Thus, we obtain the following lemma:

\begin{lemma}[The cell oscillating-coefficient problem]\label{lem3.3}\hfill\\
The constant matrix of effective coefficients is determined by averaging 
\begin{equation}\label{3.18}
\begin{split}
&A^0 :=\bigl\langle D \bigl(N(y) +y\bigr) A\bigr\rangle_y \in \mathbb{R}^{d\times d}. 
\end{split}
\end{equation}
Moreover, $A^0$ is a symmetric and positive definite matrix with the entries: 
\begin{multline}\label{3.19}
A^0_{ij} =\Bigl\langle \sum_{k,l=1}^d (N_{i,k} +\delta_{i,k}) A_{kl} 
(N_{j,l} +\delta_{j,l}) \Bigl\rangle_{\!y} 
+{\textstyle\frac{1}{|\Upsilon|}}
\int_{\partial\omega} \alpha [\![N_i]\!] [\![N_j]\!] \,dS_y\\
\text{for}\quad  i,j=1,\dots,d.
\end{multline}
For the transformed solution vector $N^\eps(x) :=N(\{\frac{x}{\eps}\})$, 
which depends only on $\{\frac{x}{\eps}\}$ since  
the coefficient $A^\eps(x) := 
A(\{\frac{x}{\eps}\})$ also depends only on $\{\frac{x}{\eps}\}$, 
the following decomposition holds:
\begin{equation}\label{3.17}
\begin{split}
&D (\eps N^\eps(x) +x) A^\eps(x) =A^0 +\eps B^\eps(x)
\quad\text{in $\mathbb{R}^{d\times d}$} \ \text{and} \ 
\text{a.e.}\;
x\in \Omega\setminus\partial\omega_\#.
\end{split}
\end{equation}
The transformed function $B^\eps(x) := B(\{\frac{x}{\eps}\})$ 
is deduced from the symmetric matrix 
$B\in L^2_{\rm div}(\Upsilon\setminus\partial\omega)^{d\times d}$ 
with zero average $\langle B\rangle_y=0$. 
Its entries $B_{ij}(y)$, $i,j=1,\dots,d$ express divergence free 
fields (called solenoidal in 3d) obtained by combining 
the derivatives $\frac{\partial}{\partial y_k}$, $k=1,\dots,d$ 
of a third-order skew-symmetric tensor $b_{ijk}$ in the following way
\begin{equation}\label{3.20}
\begin{split}
&B_{ij} =\sum_{k=1}^d b_{ijk,k},\qquad b_{ijk} =-b_{ikj},
\;\text{(skew-symmetry)} 
\quad\text{a.e. on $\Upsilon\setminus\partial\omega$}�. 
\end{split}
\end{equation} 
It follows in particular from \eqref{3.20} that 
\begin{equation}\label{3.21}
\begin{split}
&\sum_{j,k=1}^d b_{ijk}=0, \quad \sum_{j=1}^d B_{ij,j} =0,
\quad i=1,\dots,d\quad
\text{a.e. on $\Upsilon\setminus\partial\omega$}. 
\end{split}
\end{equation} 
At the interface the following jump relations hold: 
\begin{equation}\label{3.22}
\begin{split}
&[\![B^\eps]\!] =0,\qquad 
(A^0 +\eps B^\eps)\nu =\alpha [\![N^\eps]\!]
\qquad
\text{a.e. on $\partial\omega_\#$}.
\end{split}
\end{equation} 
\end{lemma}

\begin{proof}
The constant values of $A^0$ stated in  \eqref{3.18} follow from averaging \eqref{3.23} with $\langle\,\cdot\,\rangle_y$ over 
$\Upsilon\setminus\partial\omega$ and by using $\langle B\rangle_y=0$.
The formula \eqref{3.19} can be checked directly. The symmetry and positive definiteness of 
$A^0$ follow straightforward from the assumption in \eqref{2.6} of $A$ being symmetric and positive definite. 
The formulas \eqref{3.20} and \eqref{3.21} describe the fact that the columns of $B$ are divergence free. 
Inserting the representation \eqref{3.23} into \eqref{3.16} and integrating by parts yields
\begin{equation*}
\begin{split}
0 &=\int_{\Upsilon\setminus\partial\omega} (A^0+B) \nabla u \,dy
+\int_{\partial\omega} \alpha [\![N]\!] [\![u]\!] \,dS_y\\
&=\int_{\partial\omega} \bigl( \alpha [\![N]\!] [\![u]\!] 
-[\![(A^0+B)\nu\,u]\!] \bigr) \,dS_y
\end{split}
\end{equation*}
due to the second equality in \eqref{3.21}. 
Then, by choosing test-functions $u\in H^1_\#(\Upsilon\setminus\partial\omega)$ 
satisfying either $[\![u]\!]=0$ or $[\![u]\!]\not=0$, it follows  
\begin{equation}\label{3.24}
\begin{split}
&[\![B]\!] =0,\qquad 
(A^0 +B)\nu =\alpha [\![N]\!]\qquad
\text{a.e.}\
\text{on $\partial\omega$}.
\end{split}
\end{equation} 

Finally, we apply the periodic coordinate transformation 
$y\mapsto x$, $\Upsilon\mapsto\mathbb{R}^d$, with $y=\{\frac{x}{\eps}\}$ 
to \eqref{3.23} and \eqref{3.24}. 
With $\nabla_{\!y}\mapsto \eps\nabla_{\!x}$, we have for the row-wise gradient matrix
$D_y N\mapsto \eps D_x N^\eps$ and $B\mapsto \eps B^\eps$.
Thus, we arrive at \eqref{3.17} and \eqref{3.22}. 
The proof is completed. 
\end{proof}

\subsection{The main Theorem}\label{sec3.2}

Based on the Lemmata \ref{lem3.1}--\ref{lem3.3}, we formulate  
the main homogenisation result: 

\begin{theorem}\label{theo3.1}
The homogenisation of the discontinuous nonlinear PB problem under 
the interfacial transmission conditions \eqref{2.13} yields 
the following averaged (macroscopic) nonlinear PB problem: 
Find $\phi^0\in H^1_0(\Omega)$ such that 
\begin{multline}\label{3.25}
\int_\Omega \bigl(  (\nabla\phi^0)^\top A^0\nabla\phi 
-{\textstyle\frac{|\Upsilon\setminus\omega|}{|\Upsilon|}}\sum_{s=0}^n z_s 
e^{-{\textstyle\frac{z_s}{{\kappa}T}} \phi^0} \phi \bigr)\,dx
=\int_\Omega {\textstyle\frac{|\partial\omega|}{|\Upsilon|}} g \phi \,dx\\
\text{for all test-functions $\phi\in H^1_0(\Omega)$}.
\end{multline}
In the limit $\eps\searrow0^+$, the solution $\phi^\eps$ of \eqref{2.13} 
converges strongly to the first order asymptotic approximation 
$\phi^1 :=\phi^0 +\eps (\nabla\phi^0)^\top N^\eps$. This corrector term  
to $\phi^0$ satisfies the residual  
error estimate (improving \eqref{2.18}):
\begin{equation}\label{3.26}
\begin{split}
&\|\nabla (\phi^\eps -\phi^1)\|_{L^2(\Omega\setminus\partial\omega_\#)}^{2} 
+{\textstyle\frac{1}{\eps}}
\|[\![\phi^\eps-\phi^1]\!]\|_{L^2(\partial\omega_\#)}^{2} 
={\rm O} (\eps).
\end{split}
\end{equation}
\end{theorem}

\begin{proof}
First, we remark that the left hand side of \eqref{3.26} defines a norm 
in $H^1(\Omega\setminus\partial\omega_\#)$ 
due to the lower estimate \eqref{2.19}. 

Secondly, the unique solution $\phi^0$ of \eqref{3.25} can be establish by following the arguments given in the proof of Theorem~\ref{theo2.1}. 
Moreover, the solution is smooth inside $\Omega$ by standard 
arguments of local regularity of weak solutions, 
see \cite{LC/06} and references therein. 

Next, we prove the residual error estimate \eqref{3.26}. 
Integrating \eqref{3.25} by parts on $\Omega$ yields the strong formulation 
\begin{equation}\label{47a}
-{\rm div} \bigl((\nabla\phi^0)^\top A^0\bigr) 
-{\textstyle\frac{|\Upsilon\setminus\omega|}{|\Upsilon|}}\sum_{s=0}^n z_s 
e^{-{\textstyle\frac{z_s}{{\kappa}T}} \phi^0} 
={\textstyle\frac{|\partial\omega|}{|\Upsilon|}} g,\qquad \text{in $\Omega$}.
\end{equation}
By applying the Green formulas 
\eqref{13a} and \eqref{13b}
in $\Omega\setminus\omega_\#$ and $\omega_\#$, 
respectively, we have for all $\phi\in H^1(\Omega\setminus\omega_\#)$: 
$\phi=0$ on $\partial\Omega$
\begin{equation*}
\int_{\Omega\setminus\omega_\#} \!\!\! (\nabla\phi^0)^{\!\top} A^0 \nabla \phi\,dx
=-\int_{\Omega\setminus\omega_\#} \!\!\! 
\phi\,{\rm div}\bigl((\nabla\phi^0)^{\!\top} A^0\bigr)\,dx
-\int_{\partial\omega_\#^+} (\nabla\phi^0)^{\!\top} A^0 \phi\nu \,dS_x,
\end{equation*}
and for all $\phi\in H^1(\omega_\#)$: 
\begin{equation*}
\int_{\omega_\#} \!\!\! (\nabla\phi^0)^{\!\top} A^0 \nabla \phi\,dx
=-\int_{\omega_\#} \!\!\! 
\phi\,{\rm div} \bigl((\nabla\phi^0)^{\!\top} A^0\bigr)\,dx
+\int_{\partial\omega_\#^-} (\nabla\phi^0)^{\!\top} A^0 \phi\nu \,dS_x.
\end{equation*}
By summing these two expressions and by using the continuity of $\nabla\phi^0$ across the interface $\partial\omega_\#$, 
we insert the strong formulation \eqref{47a} into the above right hand sides and
rewrite problem \eqref{3.25} in the disjoint domain 
$\Omega\setminus\partial\omega_\#$ as follows 
\begin{multline}\label{3.27}
\int_{\Omega\setminus\partial\omega_\#} 
\bigl(  (\nabla\phi^0)^\top A^0\nabla\phi 
-{\textstyle\frac{|\Upsilon\setminus\omega|}{|\Upsilon|}}\sum_{s=0}^n z_s 
e^{-{\textstyle\frac{z_s}{{\kappa}T}} \phi^0} \phi \bigr)\,dx\\
+\int_{\partial\omega_\#}  (\nabla\phi^0)^\top A^0\nu [\![\phi]\!] \,dS_x
=\int_{\Omega\setminus\partial\omega_\#} 
{\textstyle\frac{|\partial\omega|}{|\Upsilon|}} g \phi \,dx\\
\text{for all test-functions $\phi\in H^1(\Omega\setminus\partial\omega_\#)$: 
$\phi=0$ on $\partial\Omega$}.
\end{multline}

In the following, we expand the terms in \eqref{3.27} based on the 
Lemmata~\ref{lem3.1}--\ref{lem3.3}. 
By applying the decomposition \eqref{3.17} of Lemma~\ref{lem3.3} 
to the integrand of the first term in the left 
hand side of \eqref{3.27}, we can represent it as the following sum 
\begin{equation}\label{3.28}
\begin{split}
(\nabla\phi^0)^\top A^0 \nabla\phi =&\ (\nabla\phi^0)^\top 
\bigl( (\eps D N^\eps +I) A^\eps -\eps B^\eps \bigr) \nabla\phi\\
=&\ \Bigl[ 
\Bigl(
\nabla \bigl( \phi^0 +\eps (\nabla\phi^0)^\top N^\eps \bigr)
\Bigr)^{\!\top} 
A^\eps -\eps (N^\eps)^\top D(\nabla\phi^0) A^\eps\\
&\ -(\nabla\phi^0)^\top\eps B^\eps \Bigr] \nabla\phi,
\end{split}
\end{equation}
where we have used 
that  
$\bigl[
\nabla \bigl( (\nabla\phi^0)^\top N^\eps \bigr) \bigr]^\top 
=(\nabla\phi^0)^\top DN^\eps +(N^\eps)^\top D(\nabla\phi^0).$

Next, the integral of the last function on the right hand side of \eqref{3.28} 
can be integrated by parts by using \eqref{3.20} and \eqref{3.21} to calculate
\begin{multline}\label{3.29}
-\int_{\Omega\setminus\partial\omega_\#} 
(\nabla\phi^0)^\top \eps B^\eps \nabla\phi \,dx
=\int_{\Omega\setminus\partial\omega_\#} 
\sum_{i,j,k=1}^d \phi^0_{,ij} \eps b_{ijk,k}^\eps \phi \,dx\\ 
+\int_{\partial\omega_\#}  
\sum_{i,j,k=1}^d \phi^0_{,i} \eps b_{ijk,k}^\eps \nu_j
[\![\phi]\!] \,dS_x
=-\int_{\Omega\setminus\partial\omega_\#} 
\sum_{i,j,k=1}^d \phi^0_{,ij} \eps b_{ijk}^\eps \phi_{,k} \,dx\\ 
+\int_{\partial\omega_\#}  \bigl( 
(\nabla\phi^0)^\top \eps B^\eps \nu 
-\sum_{i,j,k=1}^d \phi^0_{,ij} b_{ijk}^\eps \nu_k \bigr)
[\![\phi]\!] \,dS_x, 
\end{multline}
with $b_{ijk}^\eps(x) :=b_{ijk}(\{\frac{x}{\eps}\})$.
Substituting \eqref{3.28} and \eqref{3.29} in \eqref{3.27},  we rewrite it 
\begin{multline}\label{3.30}
\int_{\Omega\setminus\partial\omega_\#} \!
\Bigl[ 
\Bigl(
\nabla \bigl( \phi^0 +\eps (\nabla\phi^0)^\top N^\eps \bigr)
\Bigr)^{\!\top} 
A^\eps \nabla\phi 
-{\textstyle\frac{|\Upsilon\setminus\omega|}{|\Upsilon|}} \sum_{s=0}^n z_s 
e^{-{\textstyle\frac{z_s}{{\kappa}T}} \phi^0} \phi \Bigr]\,dx\\
+\int_{\partial\omega_\#}  (\nabla\phi^0)^\top 
(A^0 +\eps B^\eps) \nu [\![\phi]\!] \,dS_x
=\int_{\Omega\setminus\partial\omega_\#} 
{\textstyle\frac{|\partial\omega|}{|\Upsilon|}} g \phi \,dx\\
+\eps m_{\Omega\setminus\partial\omega_\#} 
\bigl(  D(\nabla\phi^0),\nabla\phi \bigr)
+m_{\partial\omega_\#} \bigl(  D(\nabla\phi^0),[\![\phi]\!] \bigr),
\end{multline}
where the bilinear continuous forms are given by 
\begin{subequations}\label{3.31}
\begin{multline}\label{3.31a}
m_{\Omega\setminus\partial\omega_\#} 
\bigl(  D(\nabla\phi^0),\nabla\phi \bigr)\\
 :=\int_{\Omega\setminus\partial\omega_\#}  \bigl( 
(N^\eps)^\top D(\nabla\phi^0) A^\eps \nabla\phi
+\sum_{i,j,k=1}^d \phi^0_{,ij} b_{ijk}^\eps \phi_{,k} \bigr) \,dx,
\end{multline}
\begin{equation}\label{3.31b}
m_{\partial\omega_\#} \bigl(  D(\nabla\phi^0),[\![\phi]\!] \bigr) 
:=\int_{\partial\omega_\#}  
\sum_{i,j,k=1}^d \phi^0_{,ij} b_{ijk}^\eps \nu_k [\![\phi]\!] \,dS_x. 
\end{equation}
\end{subequations}

Next, we apply Lemma~\ref{lem3.2} with 
$f(x)=-\sum_{s=0}^n z_s \exp({-{\textstyle\frac{z_s}{{\kappa}T}} \phi^0(x)})$
and obtain the following representation of the nonlinear term 
in \eqref{3.30} 
\begin{multline}
-{\textstyle\frac{|\Upsilon\setminus\omega|}{|\Upsilon|}}
\sum_{s=0}^n\int_{\Omega\setminus
\partial\omega_\#} 
z_s e^{-{\textstyle\frac{z_s}{{\kappa}T}} \phi^0(x)}\phi \,dx\\ 
=-\sum_{s=0}^n \int_{\Omega\setminus
\omega_\#}
z_s e^{-{\textstyle\frac{z_s}{{\kappa}T}} \phi^0(x)}\phi\,dx 
+\eps\, l_2(\phi).
\end{multline}
The boundary integral in \eqref{3.30} 
can be expanded by using \eqref{3.2} in Lemma~\ref{lem3.1}, i.e. 
\begin{equation*}
\int_{\Omega\setminus\partial\omega_\#} 
{\textstyle\frac{|\partial\omega|}{|\Upsilon|}} g\phi\,dx
=\int_{\partial\omega_\#^-} \eps g \phi^- \,dS_x
-\eps\, l_1(\phi),
\end{equation*}

Next, we subtract the equation \eqref{3.30} for $\phi^0$ 
from the perturbed equation \eqref{2.13b} for $\phi^\eps$ 
and use the notation 
$\phi^1 :=\phi^0 +\eps (\nabla\phi^0)^\top N^\eps$. 
Moreover, for $\phi^1$, we remark that $[\![\phi^0]\!]=0$ 
at $\partial\omega_\#$. Hence 
${\textstyle\frac{\alpha}{\eps}} [\![\phi^1]\!]  
=\alpha (\nabla\phi^0)^\top [\![N^\eps]\!] 
=(\nabla\phi^0)^\top (A^0 +\eps B^\eps) \nu$ 
in view of \eqref{3.22}. Thus, after subtracting \eqref{3.30}
from \eqref{2.13b},
we calculate using the above relations
\begin{multline}\label{3.32}
\int_{\Omega\setminus\partial\omega_\#} \!\!\!
\nabla(\phi^\eps -\phi^1)^\top A^\eps\nabla\phi \,dx
+\int_{\partial\omega_\#} {\textstyle\frac{\alpha}{\eps}} 
[\![\phi^\eps -\phi^1]\!] [\![\phi]\!] \,dS_x\\
-\sum_{s=0}^n \int_{{\Omega\setminus\omega_\#}} \!\!\! 
z_s \bigl( e^{-{\textstyle\frac{z_s}{{\kappa}T}} \phi^\eps} 
-e^{-{\textstyle\frac{z_s}{{\kappa}T}} \phi^0} \bigr) \phi \,dx
=\eps (l_1(\phi)+l_2(\phi))\\ 
-\eps\, m_{\Omega\setminus\partial\omega_\#} 
\bigl(  D(\nabla\phi^0),\nabla\phi \bigr)
-m_{\partial\omega_\#} \bigl(  D(\nabla\phi^0),[\![\phi]\!] \bigr).
\end{multline}

One difficulty is that $\phi^1$ cannot be substituted as test-function 
into \eqref{3.32} since $\phi^1\not=0$ at the boundary $\partial\Omega$. 
For its lifting, we take a cut-off function $\eta_\eps$ supported in 
a $\eps$-neighborhood of $\partial\Omega$ such that $\eta_\eps=1$ at 
$\partial\Omega$. 
Hence, $\nabla\eta_\eps\sim{\textstyle\frac{1}{\eps}}$ and 
${\rm supp}(\eta_\eps)\sim\eps$. 
Due to the assumed $\eps$-gap between $\partial\Omega$ and $\omega_\#$, 
we remark that ${\rm supp}(\eta_\eps)$ does not intersect $\partial\omega_\#$. 

After substitution of $\phi =\phi^\eps -\phi^1_{\eta_\eps}$ with 
$\phi^1_{\eta_\eps} :=\phi^0 +\eps (1-\eta_\eps) 
(\nabla\phi^0)^\top N^\eps$ into \eqref{3.32} 
and by using $[\![\phi^1_{\eta_\eps}]\!] =[\![\phi^1]\!]$, 
we obtain the equality 
\begin{align}
\int_{\Omega\setminus\partial\omega_\#} &
\nabla(\phi^\eps -\phi^1)^\top A^\eps\nabla (\phi^\eps -\phi^1) \,dx
+\int_{\partial\omega_\#} {\textstyle\frac{\alpha}{\eps}} 
[\![\phi^\eps -\phi^1]\!]^2 \,dS_x\nonumber\\
&-\sum_{s=0}^n \int_{\Omega\setminus\omega_\#} 
z_s \bigl( e^{-{\textstyle\frac{z_s}{{\kappa}T}} \phi^\eps} 
-e^{-{\textstyle\frac{z_s}{{\kappa}T}} \phi^1_{\eta_\eps}} \bigr) 
(\phi^\eps -\phi^1_{\eta_\eps}) \,dx\nonumber\\ 
=& -m_{\eta_\eps} \bigl( 
\nabla (\phi^\eps -\phi^1),D(\nabla\phi^0)\bigr) 
-m_{\partial\omega_\#} \bigl(  D(\nabla\phi^0), 
[\![\phi^\eps -\phi^1]\!] \bigr)\nonumber\\ 
&+\eps\,  \widetilde{l}(\phi^\eps -\phi^1_{\eta_\eps}) ,\label{3.33}
\end{align}
where we introduce the form $m_{\eta_\eps}$ 
due to the cut-off function as 
\begin{multline}\label{3.35}
m_{\eta_\eps} \bigl( \nabla (\phi^\eps -\phi^1),D(\nabla\phi^0)  \bigr)\\
:=\eps \int_{{\rm supp}(\eta_\eps)} \!\!\! \!\!\!
\nabla (\phi^\eps -\phi^1)^\top A^\eps \nabla \bigl( 
\eta_\eps (\nabla\phi^0)^\top N^\eps \bigr) \,dx, 
\end{multline}
and the short notation $\widetilde{l}$ stands for the following terms
\begin{equation}\label{3.34}
\begin{split}
&\widetilde{l}(\phi) :=l_1(\phi)+l_2(\phi) 
-m_{\Omega\setminus\partial\omega_\#} \bigl( D(\nabla\phi^0),\nabla \phi \bigr) 
+m^\eps (\phi^0,\phi),
\end{split}
\end{equation}
where the nonlinear form $m^\eps$ in \eqref{3.34} 
is given by 
\begin{equation}\label{3.36}
m^\eps (\phi^0,\phi)
:=\sum_{s=0}^n \int_{\Omega\setminus\omega_\#} \!\!\! 
z_s e^{-{\textstyle\frac{z_s}{{\kappa}T}} \phi^0}
{\textstyle\frac{1}{\eps}} 
\bigl( 1 -e^{-\eps (1-\eta_\eps) 
{\textstyle\frac{z_s}{{\kappa}T}} (\nabla\phi^0)^\top N^\eps} \bigr) \phi \,dx.
\end{equation}
From \eqref{3.36}, it can be estimated uniformly as 
\begin{equation}\label{3.37}
\begin{split}
&\bigl| m^\eps (\phi^0,\phi) \bigr|
\le K \|\nabla \phi\|_{L^2(\Omega\setminus\omega_\#)}
\le K \|\nabla\phi\|_{L^2(\Omega\setminus\partial\omega_\#)}, \qquad (K>0),
\end{split}
\end{equation}
due to the Taylor series 
$1 -e^{-\eps \xi} =\eps \xi +o(\eps)$ for small $\eps$. 

The left hand side of \eqref{3.33} can be estimated from below by applying 
the coercivity of the matrix $A$ as assumed in \eqref{2.6} 
and by observing that the third term on the left hand side is nonnegative due to the strict monotonicity of the exponential function.
Altogether with \eqref{2.19}, this implies that
\begin{multline}\label{3.37a}
K_5 \|\phi^\eps -\phi^1\|_{H^1(\Omega\setminus\partial\omega_\#)}^2
\le \bigl| m_{\eta_\eps} \bigl(  D(\nabla\phi^0), 
\nabla (\phi^\eps -\phi^1) \bigr) \bigr|\\
+\bigl| m_{\partial\omega_\#} \bigl(  D(\nabla\phi^0), 
[\![\phi^\eps -\phi^1]\!] \bigr) \bigr|
+\eps |\widetilde{l}(\phi^\eps -\phi^1_{\eta_\eps})|,
\end{multline}
with $K_5 =K_0(\underline{K} +\alpha)>0$ after recalling $\underline{K}$ from \eqref{2.6} and 
$K_0$ from \eqref{2.19}. 

At this point, we remark that the right-hand side of \eqref{3.37a} is a homogeneous 
function of degree one with respect to the norm 
$\|\phi^\eps -\phi^1\|_{H^1(\Omega\setminus\partial\omega_\#)}$ 
as the following estimates will prove.
Thus, the inequality \eqref{3.37a} implies directly that the norm 
$\|\phi^\eps -\phi^1\|_{H^1(\Omega\setminus\partial\omega_\#)}$ is bounded, which  
reconfirms estimate \eqref{2.18}. 

However, the following argument allows to refine the asymptotic residual estimate to obtain 
\eqref{3.26} as $\eps\searrow0^+$. 
In particular, we shall estimate the three terms at the right hand side 
of \eqref{3.37a} and then apply Young's inequality to obtain sums of sufficiently small terms 
of order $O(\|\phi^\eps -\phi^1\|^2_{H^1(\Omega\setminus\partial\omega_\#)})$ 
and constant terms, which will constitute the refined residual estimate.  
 
At first, from the estimates \eqref{3.3}, \eqref{3.8}, \eqref{3.37} and due to 
the boundedness of the bilinear form \eqref{3.31a}
for $\phi\in H^1(\Omega\setminus\partial\omega_\#)$, it follows that
\begin{equation}\label{3.38}
\begin{split}
&|\widetilde{l}(\phi)|
\le K \|\phi\|_{H^1(\Omega\setminus\partial\omega_\#)}, \qquad (K>0). 
\end{split}
\end{equation}
Since $\phi^1_{\eta_\eps} =\phi^1 -\eps \eta_\eps 
(\nabla\phi^0)^\top N^\eps$, we estimate that 
\begin{equation}\label{3.39}
\begin{split}
&\|\phi^\eps -\phi^1_{\eta_\eps} \|_{H^1(\Omega\setminus\partial\omega_\#)}^2
\le 2\|\phi^\eps -\phi^1\|_{H^1(\Omega\setminus\partial\omega_\#)}^2 
+{\rm O}(\eps). 
\end{split}
\end{equation}
Therefore, specifically for $\phi=\phi^\eps -\phi^1_{\eta_\eps}$, 
and by using Young's inequality,
it follows from \eqref{3.38} and \eqref{3.39} that
\begin{equation}\label{3.40}
\begin{split}
&|\widetilde{l}(\phi^\eps -\phi^1_{\eta_\eps})|
\le K_6 \bigl( 
\|\phi^\eps -\phi^1\|_{H^1(\Omega\setminus\partial\omega_\#)}^2 
+1 \bigr),\qquad (K_6>0).
\end{split}
\end{equation}
For $\phi\in H^1(\Omega\setminus\partial\omega_\#)$,  
by using again  
Young's inequality and 
by recalling the properties of the cut-off function $\eta_\eps$ implying 
$\int_{{\rm supp}(\eta_\eps)} |\nabla \eta_\eps|^2\,dx 
={\rm O}({\textstyle\frac{1}{\eps}})$,
we estimate \eqref{3.35} with an arbitrary $t_1\in\mathit{R}_+$ by 
\begin{equation}\label{3.41}
\begin{split}
&\bigl| m_{\eta_\eps} \bigl(  \nabla \phi ,D(\nabla\phi^0) \bigr) \bigr|
\le \eps t_1 K_7 +{\textstyle\frac{1}{t_1}} 
\|\nabla\phi\|_{L^2(\Omega\setminus\partial\omega_\#)}^2,\qquad (K_7>0),
\end{split}
\end{equation}
and the form in \eqref{3.31b} by 
\begin{equation}\label{3.42}
\begin{split}
&\bigl| m_{\partial\omega_\#} \bigl(  D(\nabla\phi^0),[\![\phi]\!] 
\bigr) \bigr| \le \eps t_2 K_8 +{\textstyle\frac{1}{\eps t_2}} 
\|[\![\phi]\!]\|_{L^2(\partial\omega_\#)}^2,\qquad (K_8>0),
\end{split}
\end{equation}
with an arbitrary $t_2\in\mathit{R}_+$. 
Therefore, by applying the estimates \eqref{3.40},  
\eqref{3.41} and \eqref{3.42} with $\phi=\phi^\eps -\phi^1$ to \eqref{3.37a} 
and for suitable $t_1,t_2$, 
and $\eps_0>0$ such that 
$$
0<K:=K_5 -({\textstyle\frac{1}{t_1}}  
+{\textstyle\frac{1}{t_2}} ) K_0 -\eps_0 K_6, 
$$ 
we conclude 
\begin{equation*}
K \|\phi^\eps -\phi^1\|_{H^1(\Omega\setminus\partial\omega_\#)}^2
\le \eps (t_1 K_7 +t_2 K_8 +K_6),
\end{equation*}
for all $\eps<\eps_0$, which yields estimate
\eqref{3.26}. This finishes the proof.
\end{proof}

\section{Discussion}\label{sec4}

In the following, we shall summarise the main observations concerning the presented results.

\begin{itemize}
\item
We remark at first that Theorem~\ref{theo3.1}, in particular, implies by standard arguments  
the weak convergence $\phi^\eps\rightharpoonup\phi^0$ in 
$H^1(\Omega\setminus\partial\omega_\#)$ and the strong convergence 
$\phi^\eps\to\phi^0$ in $L^2(\Omega\setminus\partial\omega_\#)$ 
as $\eps\searrow0^+$, 
as well as the two-scale convergence and the $\Gamma$-convergence of the solutions. 

\item
We observe that the first two terms on the right hand side of 
\eqref{3.37a} 
express the residual error near $\partial\Omega$ and at $\partial\omega_\#$. 
These terms are asymptotically of order $O(\sqrt{\eps})$
(as can be see by setting $t_1=O(\eps^{-1/2})=t_2$
in \eqref{3.41} and \eqref{3.42}) and thus constitute the leading order $O(\eps)$ 
in the residual error estimate \eqref{3.26}.

Therefore, by constructing corrector terms in form of the respective boundary layers, 
the $O(\eps)$-estimate \eqref{3.26} could be improved to the order 
$O(\eps^2)$. 

\item
The factor ${\textstyle\frac{1}{\eps}}$ appears at the jump across interface 
$\partial\omega_\#$ in the left hand side of microscopic equation \eqref{2.13b}. 
It is controlled by the coercivity condition \eqref{2.19}. 
We point out that this term disappears in the homogenisation limit and does not contribute 
to the macroscopic equation \eqref{3.25}. 

\item
The factor $\eps$ in front of the inhomogeneous material parameter $g$, which is prescribed 
at the solid phase boundary $\partial\omega_\#^-$, presents the critical order. 
After averaging this factor guarantees the presence of the potential 
${\textstyle\frac{|\partial\omega|}{|\Upsilon|}} g$ 
distributed over the homogeneous domain $\Omega$ in \eqref{3.25}. 

\item 
For variable functions $g(\{\frac{x}{\varepsilon}\})$ distributed periodically 
over the interface $\partial\omega_\#$, the decomposition 
\begin{equation*}
g =\langle g\rangle_y +G,\quad \text{with} \; \langle g\rangle_y 
:=\frac{1}{|\partial\omega|} \int_{\partial\omega} g(y)\,dy, \quad 
\langle G\rangle_y=0,
\end{equation*}
yields in the limit $\varepsilon\searrow0^+$ that the constant value $\langle g\rangle_y$ 
replaces $g$ in the averaged problem \eqref{3.25}, see e.g. \cite{CD/88}. 

\item
The nonlinear term appearing in \eqref{3.25} scales 
with the porousity coefficient 
${\textstyle\frac{|\Upsilon\setminus\omega|}{|\Upsilon|}}$. 

\end{itemize}

\vspace{5mm}\noindent {\bf Acknowledgments}. 
The research results were obtained with the support of 
the Austrian Science Fund (FWF) in the framework of 
the SFB F32 "Mathematical Optimization and Applications in
Biomedical Sciences" and project P26147-N26.
The authors gratefully acknowledge partial support by NAWI Graz.

\end{document}